\documentclass[12pt,fleqn]{amsart}

\usepackage{amssymb}
\usepackage{amsmath}
\usepackage{amsfonts}

\usepackage{color}
\usepackage{enumerate}
\newtheorem{theorem}{Theorem}[section]
\newtheorem{lemma}[theorem]{Lemma}
\newtheorem{proposition}[theorem]{Proposition}
\newtheorem{problem}[theorem]{Problem}

\newtheorem{corollary}[theorem]{Corollary}

\theoremstyle{definition}
\newtheorem{definition}[theorem]{Definition}
\newtheorem{example}[theorem]{Example}

\newtheorem{axiom}[theorem]{Axiom}
\theoremstyle{remark}
\newtheorem{remark}[theorem]{Remark}
\newtheorem{notation}[theorem]{Notation}
\numberwithin{equation}{section}

\begin{document}

%\begin{flushright} \textsf{PRELIMINARY VERSION}
%\end{flushright}
%\vskip1cm

\title[Extension operators and twisted sums]{Extension operators and twisted sums\\ of $c_0$ and $C(K)$ spaces}
%Reflecting properties of compact spaces\\ in small continuous images}

\author[W.\ Marciszewski]{Witold Marciszewski}
\author[G.\ Plebanek]{Grzegorz Plebanek}
\address{Institute of Mathematics\\
University of Warsaw\\ Banacha 2\newline 02--097 Warszawa\\
Poland} \email{wmarcisz@mimuw.edu.pl}

\address{Mathematical Institute\\ University of  Wroc\l aw\\ pl.\ Grunwaldzki 2/4\\
\newline 50--384 Wroc\-\l aw, Poland}

 \email{grzes@math.uni.wroc.pl}

\date{\today}

\thanks{The first author was partially supported by the Polish National Science Center research grant DEC-2012/07/B/ST1/03363. The second author was partially supported by the Polish National Science Center research grant NCN  2013/11/B/ST1/03596 (2014-2017).}


\begin{abstract}
We investigate the following problem posed by Cabello Sanch\'ez, Castillo, Kalton, and Yost:

Let $K$ be a nonmetrizable compact space. Does there exist a nontrivial twisted sum of $c_0$ and $C(K)$, i.e., does there exist a Banach space $X$ containing a non-complemented copy $Z$ of $c_0$ such that the quotient space $X/Z$ is isomorphic to $C(K)$?

Using additional set-theoretic assumptions we give the first examples of compact spaces $K$ providing  a negative answer to this question.
We show that under Martin's axiom and the negation of the continuum hypothesis, if either $K$ is the Cantor cube $2^{\omega_1}$ or $K$ is a separable scattered compact space of height $3$ and weight $\omega_1$, then every
twisted sum of $c_0$ and $C(K)$ is trivial.

We also construct nontrivial twisted sums of $c_0$ and $C(K)$ for
$K$ belonging to several classes of compacta.
Our main tool is an investigation of pairs of compact spaces $K\subseteq L$ which do not admit an extension operator $C(K)\to C(L)$.
\end{abstract}

\maketitle

%%%%%%%%%%%%%%%%%%%%%%%%%%%%%%%%%%%%%%%%%LOCAL SHORTENINGS
%%%%%%%%%%%%%%NUMBERS
\newcommand{\con}{\mathfrak c}
\newcommand{\eps}{\varepsilon}
%%%%%%%%%%%%%%%%%%%%FRAK FAMILIES
\newcommand{\fA}{\mathfrak A}
\newcommand{\fB}{\mathfrak B}
\newcommand{\fC}{\mathfrak C}
\newcommand{\fM}{\mathfrak M}
\newcommand{\bP}{\mathbb P}
\newcommand{\bD}{\mathbb D}
\newcommand{\bU}{\mathbb U}
\newcommand{\bG}{\mathbb G}
\newcommand{\bB}{\mathbb B}
%%%%%%%%%%%%%%%%%%%%SCRIPT FAMILIES
\newcommand{\BB}{\protect{\mathcal B}}
\newcommand{\cA}{\mathcal A}
\newcommand{\cC}{{\mathcal C}}
\newcommand{\cF}{{\mathcal F}}
\newcommand{\FF}{{\mathcal F}}
\newcommand{\GG}{{\mathcal G}}
\newcommand{\II}{{\mathcal I}}
\newcommand{\LL}{{\mathcal L}}
\newcommand{\NN}{{\mathcal N}}
\newcommand{\UU}{{\mathcal U}}
\newcommand{\VV}{{\mathcal V}}
\newcommand{\HH}{{\mathcal H}}
\newcommand{\DD}{{\mathcal D}}
\newcommand{\RR}{\protect{\mathcal R}}
\newcommand{\ide}{\mathcal N}
%%%%%%%%%%%%%%%%%%%%%%%SYMBOLS
\newcommand{\btu}{\bigtriangleup}
\newcommand{\hra}{\hookrightarrow}
\newcommand{\ve}{\vee}
\newcommand{\we}{\cdot}
\newcommand{\de}{\protect{\rm{\; d}}}
\newcommand{\er}{\mathbb R}
\newcommand{\qu}{\mathbb Q}
\newcommand{\supp}{{\rm supp} }
\newcommand{\card}{{\rm card} }
\newcommand{\wn}{{\rm int} }
\newcommand{\ult}{{\rm ult}}
\newcommand{\vf}{\varphi}
\newcommand{\osc}{{\rm osc}}
\newcommand{\cov}{{\rm cov}}
\newcommand{\cf}{{\rm cf}}
\newcommand{\ol}{\overline}
\newcommand{\me}{\protect{\bf v}}
\newcommand{\ex}{\protect{\bf x}}
\newcommand{\stevo}{Todor\v{c}evi\'c}
\newcommand{\cc}{\protect{\mathfrak C}}
\newcommand{\scc}{\protect{\mathfrak C^*}}
\newcommand{\lra}{\longrightarrow}
\newcommand{\sm}{\setminus}
\newcommand{\uhr}{\upharpoonright}

\newcommand{\sub}{\subseteq}
\newcommand{\ms}{$(M^*)$}
\newcommand{\m}{$(M)$}
\newcommand{\MA}{\mbox{\sf MA}}
\newcommand{\CH}{\mbox{\sf CH}}
\newcommand{\clop}{\protect{\rm Clop} }
\newcommand{\fX}{\mathfrak X}
\newcommand{\fY}{\mathfrak Y}
\newcommand{\fZ}{\mathfrak Z}
\newcommand{\cde}{\protect{\rm CDE}}
\newcommand{\lep}{\rm LEP}
\newcommand{\lo}{\protect{ l}}
\newcommand{\fin}{\mbox{\it fin}}
\newcommand{\ba}{M}
\newcommand{\too}{2^{\omega_1}}
\newcommand{\la}{\langle}
\newcommand{\ra}{\rangle}
\newcommand{\dist}{\protect{\rm dist}}
\newcommand{\wl}{ (\#)}
\newcommand{\mq}{M^\qu}
%%%%%%%%%%%%%%%%%%%%%%%%%%%%%%%%%%%%%%%%%%%%%%%%%%%%%%%%%%%
\section{Introduction}

A \emph{twisted sum} of Banach spaces $Z$ and $Y$ is a short exact sequence
\[0\to Z\to X \to Y\to 0, \]
where $X$ is a Banach space and the maps are bounded linear operators.
Such a twisted sum is called \emph{trivial} if the exact sequence splits, i.e.\
if the map $Z\to X$ admits a left inverse (in other words, if the map $X\to Y$ admits a right inverse).
This is equivalent to saying that  the range of the map $Z\to X$ is complemented in $X$; in this case, $X\cong Y\oplus Z$.
We can, informally, say that $X$ is a nontrivial twisted sum of $Z$ and $Y$ if $Z$ can be isomorphically embedded onto an uncomplemented copy
$Z'$ of $X$ so that $X/Z'$ is isomorphic to $Y$. Twisted sums of Banach spaces and their connection with injectivity-like properties are discussed
in a recent monograph \cite{ASCGM}.

The classical Sobczyk theorem asserts that every isomorphic copy of $c_0$ is complemented in every separable superspace. This implies that $Z=c_0$ admits a nontrivial twisted sum with no separable Banach space $Y$. In particular, there is no nontrivial twisted sum of $c_0$ and $C(K)$, whenever $K$ is a compact metric space.
Castillo \cite{Ca16} and Correa and Tausk \cite{CT16} investigated the following problem originated in \cite{CCKY03} and \cite{CCY00}.

\begin{problem}\label{i:1}
Given  a nonmetrizable compact space $K$, does there exist  a nontrivial twisted sum of $c_0$ and $C(K)$?
\end{problem}

There are several classes of nonmetrizable compacta for which  Problem \ref{i:1} has an affirmative  answer, cf.\ \cite{Ca16},
\cite{CT16}. We prove, however, that  under Martin's axiom ($\MA$) and the negation of the continuum hypothesis ($\CH$)  there are nonmetrizable compacta $K$ with the property
that every twisted sum of $c_0$ and $C(K)$ is trivial, see Corollary \ref{tts:2} and \ref{tts:3}.

Our main results are contained in Section \ref{tts} and rely on considerations from Sections 3 and 4.
In Section \ref{ots} we investigate the following  construction that enables us to handle, for a fixed space $K$, all possible twisted sums
of $c_0$ and $C(K)$:
Let $L$ denote the unit ball in $C(K)^*$ equipped with the $weak^*$ topology; then $C(K)$ embeds
into $C(L)$ in a canonical way.
Take any compact space $L'$ which consists of $L$ and countably many isolated points; we say that such $L'$ is  a countable discrete extension of $L$.
 Then $C(L')$ contains, again in a canonical way, a twisted sum of $c_0$ and $C(K)$.
This leads to formulating a certain property of the space $K$, called the $^*$extension property,  under which every twisted sum of $c_0$ and $C(K)$ is trivial, see Theorem \ref{ts:2}.

In Section \ref{am} we consider functions on a given Boolean algebra $\fA$. We say that $\fA$ has the approximation property if every sequence
of ``nearly additive'' functions on $\fA$ can be suitably approximated by a sequence of finitely additive signed measures on $\fA$; see Definition \ref{am:-1}
for the precise statement. The second major step towards our main result is Theorem \ref{am:6} which yields  that under $\MA$ and the negation of  $\CH$  there are uncountable algebras $\fA$
having the approximation property. The proof of \ref{am:6} is quite involved and makes use of  a result from \cite{BRS94} and its consequences that are  discussed in Appendix.

Then in Section \ref{tts} we prove that a compact zerodimensional space $K$ has the $^*$extension property whenever the algebra of clopen subsets of $K$
has the approximation property. From this result we conclude, in particular,  that the space $K=2^{\omega_1}$  is a relatively consistent counterexample to the question in
Problem \ref{i:1}.  By a result due to Correa and Tausk \cite{CT16} this shows that the problem if $c_0$ admits  a nontrivial twisted sum with $C(2^{\omega_1})$
cannot be decided within the usual set theory.

For a compact space $K$, we identify $C(K)^*$ with the space $M(K)$ of signed Radon measures on $K$ with finite variation, and we denote the unit ball of $M(K)$, equipped with the $weak^*$ topology, by $M_1(K)$. The $^*$extension property of $K$ mentioned above is closely related to the existence of an extension operator
$E:C(M_1(K))\to C(L)$, where $L$ is a countable discrete extension of $M_1(K)$.
From that point of view, the ball $M_1(K)$ is more difficult to deal with than its subspace $P(K)$, consisting of all probability Radon measures on $K$.
Recall that $P(2^{\omega_1})$ is homeomorphic to $[0,1]^{\omega_1}$ so is a Dugundji space, which guarantees the existence of an extension operator $E:C(P(2^{\omega_1}))\to C(L)$ for any countable discrete extension of $M_1(K)$.
However, we show in Section 6  that the ball  $M_1(K)$ is never  a Dugundji space for nonmetrizable $K$. Roughly speaking, there seems to be no short way to
checking the $^*$extension property for $2^{\omega_1}$.

Sections 7-9 are devoted to those compacta $K$ for which one can construct a nontrivial twisted sum of $c_0$ and $C(K)$.
We explore here the following natural approach.

 For a given compact space $K$ consider a countable discrete extension $L$ of the space $K$ itself, i.e.,
a compactum $L$ consisting of $K$ and countably many isolated points.
Denote the subspace $\{f\in C(L): f|K\equiv 0\}$ by ${C(L|K)}$. Note that if $L\sm K$ is countable infinite and discrete then $C(L|K)$ is isometric to $c_0$ so
$C(L)$ is a twisted sum of $c_0$ and $C(K)$. Such a twisted sum is trivial if and only if there is an extension operator
$E:C(K)\to C(L)$, that is a bounded operator such that $Eg|K=g$ for every $g\in C(K)$.

Hence the natural way of constructing a nontrivial twisted sum of $c_0$ and $C(K)$ is to find a countable discrete extension $L$ of $K$
without a corresponding extension operator.
Building on this simple idea  we construct nontrivial twisted sums of $c_0$ and $C(K)$ for spaces $K$ from several classes of  nonmetrizable compacta:
dyadic spaces (Section \ref{dyadic}), linearly ordered compact spaces (Section \ref{lo}), scattered compacta of finite height
(Section \ref{sc}). In this way we extend results from Castillo \cite{Ca16} and Correa and Tausk \cite{CT16} or present their alternative proofs.
It turns out that the problem of constructing a countable discrete extension of a given space $K$ without the corresponding extension operator
is related to some problems of infinitary combinatorics; for instance, our Example \ref{d:4} uses multiple gaps of Avil\'es and Todorcevic \cite{AT11}.
Note also that  our Theorem \ref{added} shows that the main result of Castillo \cite{Ca16} (proved under $\CH$) does require some additional set-theoretic assumptions.

Section 10 offers additional comments and collect some related open problems.
\smallskip

The second author is  very grateful to the anonymous referee of \cite{DP16} who pointed out interesting connection of the results presented there with Problem \ref{i:1}.

\section{Preliminaries}\label{pr}

If $K$ is a compact space then $C(K)$ is the familiar Banach space of continuous real-valued functions on $K$.
We always identify $C(K)^*$ with the space $M(K)$ of signed Radon measures on $K$ with finite variation. $M_1(K)$ stands for the closed unit ball of $M(K)$, equipped with the $weak^*$ topology inherited from $C(K)^*$. We denote the ball $rM_1(K)$ of radius $r>0$ by $M_r(K)$.
The symbol $P(K)$ denotes the subspace of $M_1(K)$ consisting of all probability measures; given $x\in K$, $\delta_x\in P(K)$ is the Dirac measure, a point mass concentrated at the point $x$.

It will  be convenient to use the following notion.

\begin{definition}
If $L$ is a compact space then a compact superspace $L'\supseteq L$ will be called a {\em countable discrete extension} of $L$ if
$L'\sm L$ is countable and discrete.
\end{definition}

We shall write $L'\in \cde(L)$ to say that $L'$ is such an extension of $L$.
Typically, when $L'\sm L$ is dense in $L'$, $L'$ is homeomorphic to a compactification of $\omega$ such that its remainder is homeomorphic to $L$.
Unless stated otherwise, if $L'\in\cde(L)$ and $L'\sm L$ is infinite, then we identify $L'\sm L$ with the set of natural numbers $\omega$.

For the future reference we state the following simple observations on countable discrete extensions of compacta.

\begin{lemma}\label{ts:0.25}
If $L'\in\cde(L)$ and $h_1,h_2\in C(L')$ agree on $L$ then $\lim_n(h_1(n)-h_2(n))=0$.
\end{lemma}

\begin{proof}
Otherwise, $|h_1(n)-h_2(n)|\ge \eps$ for some $\eps>0$ and $n$ from some infinite set $N\sub\omega$. Taking an accumulation point $x$ of $N\sub L'$ we get
$x\in L$ and $h_1(x)\neq h_2(x)$, a contradiction.
\end{proof}

\begin{lemma}\label{c:3}
Let $L'\in\cde(L)$ and let $f_1,\ldots, f_k\in C(L')$ for some $k$. Then for every $\eps>0$ there is $n_0$ such that for every $n\ge n_0$
there is $x\in L$ such that  $|f_i(x)-f_i(n)|<\eps$ for every $i\le k$.
\end{lemma}

\begin{proof}
For every $x\in L$ take an open neighborhood  $V_x\sub L'$ of $x$ such that $|f_i(x)-f_i(y)|<\eps$ for every $y\in V_x$ and $i\le k$.
Then $L\sub V=\bigcup_{j\le m} V_{x_j}$ for some $m$ and $x_1,\ldots, x_m\in L$. Then $L'\sm V$ must be finite and we are done.
\end{proof}

%Recall that
%in the terminology of \cite{DP16}, a compactification $\gamma\omega$ is {\em tame} if the natural copy of $c_0$ is complemented in
%$C(\gamma\omega)$; equivalently, if there is a bounded  extension operator $E:C(\gamma\omega\!\sm\!\omega)\to C(\gamma\omega)$.

Given compact spaces $L,L'$ with $L\sub L'$, an extension operator $E:C(L)\to C(L')$ is a bounded linear operator
such that $Eg|L=g$ for every $g\in C(L)$.

%Recall the following standard facts (see e.g.\ \cite{DP16}, Corollary 2.3 and Lemma 2.4).

Let us consider $L'=L\cup \omega\in \cde(L)$ and  a short exact sequence
\[0\to c_0\stackrel{i}{\rightarrow} C(L') \stackrel{R}{\rightarrow} C(K)\to 0,\]
where $i$ is the natural isometric embedding, which  sends the unit vector $e_n\in c_0$ to the characteristic function $\chi_{\{n\}}\in C(L')$,
and $R$ is the restriction operator. Now if $E:C(L)\to C(L')$ is the extension operator then $E$ is a right inverse of $R$ so the above short sequence splits.
This leads to the following observation.

\begin{lemma}\label{p:1}
For a compact space $L$ and $L'\in\cde{(L)}$ the following are equivalent

\begin{enumerate}[(i)]
\item there is no extension operator $E:C(L)\to C(L')$;
\item $0\to c_0\stackrel{i}{\rightarrow} C(L') \stackrel{R}{\rightarrow} C(K)\to 0$
%$C(L')$ 
is a nontrivial twisted sum of $c_0$ and $C(L)$.
\end{enumerate}
\end{lemma}

%\begin{proof}
%Note that $L'\sm L$ must be infinite; we can identify it with $\omega$.
%Define an embedding $i:c_0\to C(L')$, sending
%Then $Z=i[c_0]$ is the subspace of $C(L')$ of all functions vanishing on $L$ and hence $C(L')/Z$ is isomorphic to $C(L)$.
%The subspace $Z$ of $C(L')$ is not complemented. Indeed, supposing that
%$P$ is a projection from $C(L')$ onto $Z$  one can easily define an  extension operator  $E:C(L)\to C(L')$
%by $Eg=\widehat{g} - P\widehat{g}$, where $\widehat{g}$ is {\em any} extension of $g\in C(L)$ to a continuous function on $L'$
%with the same norm. The point is that if $\widetilde{g}\in C(L')$ also extends $g\in C(L)$ then $\widehat{g}-\widetilde{g}$ vanishes
%on $L$ so $P(\widehat{g}-\widetilde{g})=\widehat{g}-\widetilde{g}$, and therefore  $E$ is well-defined.
%\end{proof}

\begin{remark}\label{p:11}
Let $L'\in\cde{(L)}$ be as in Lemma \ref{p:1} and  $K$ be a compact space containing $L$ and such that $K\cap (L'\sm L)=\emptyset$. Then we can treat $K'=K\cup (L'\sm L)$ as a countable discrete extension of $K$. If we additionally assume that there exists a extension operator $E:C(L)\to C(K)$ then it can be easily observed that there is no extension operator $E:C(K)\to C(K')$, hence  $C(K')$ is a nontrivial twisted sum of $c_0$ and $C(K)$.
\end{remark}

Following \cite{DP16} we say that a compactification $\gamma\omega$ of the discrete space $\omega$ is {\em tame} if the natural copy of $c_0$ in
$C(\gamma\omega)$, consisting of all functions from vanishing on the remainder $K=\gamma\omega\!\sm\!\omega$, is complemented in $C(\gamma\omega)$.
This is equivalent to saying that  $\gamma\omega\in \cde(K)$ and there is a corresponding extension operator.
We include an easy observation related to Lemma \ref{p:1}.

\begin{proposition}\label{p:12}
Let $L$ be a separable compact space and $L'\in\cde{(L)}$ be such that there is no extension operator $E:C(L)\to C(L')$. Then there exists a non-tame compactification $\gamma\omega$ with the remainder $\gamma\omega \sm\omega$ homeomorphic to $L$.
\end{proposition}

\begin{proof} Let $\{d_n: n\in\omega\}$ be a countable dense subset of $L$. Consider the following subset of the product $L'\times [0,1]$:
\[K= L'\times\{0\}\cup \left\{\left(d_n,\frac{1}{n+k+1}\right): k,n\in\omega\right\}.\]
Obviously, $C=(L'\sm L)\times\{0\}\cup \{(d_n,1/(n+k+1)): k,n\in\omega\}$ is a countable discrete space and  $K$ is a compactification of $C$ with the required properties.
\end{proof}
\medskip

The following standard fact reduces the problem of defining an extension operator for $L'\in\cde(L)$
 to a problem of finding a suitable sequence of measures.

\begin{lemma}\label{p:2}
Let $L'=L\cup\omega$ be a countable discrete extension of a compact space $L$. The following are equivalent

\begin{enumerate}[(i)]
\item there is a extension operator $E: C(L)\to C(L')$ with $\|E\|\le r$;
\item there is a continuous map $\varphi: L'\to M_r(L)$ such that $\varphi(x) = \delta_{x}$ for any $x\in L$;
\item there is a sequence $(\nu_n)_n$ in $M(L)$ such that $\|\nu_n\|\le r$ for every $n$ and $\nu_n-\delta_n\to 0$ in the $weak^*$ topology of $M(L')$.
\end{enumerate}
\end{lemma}

\begin{proof}
$(i)\rightarrow (ii)$. Let $\varphi(x)=E^*\delta_x$, for $x\in L'$. Clearly, the map $\varphi$ is continuous and takes values in $M_r(L)$. If $x\in L$ then, for any $f\in C(L)$,
$E^*\delta_x(f)=Ef(x)=f(x)$, hence $\varphi(x)=\delta_x$.

$(ii)\rightarrow (iii)$.
Define $\nu_n=\varphi(n)$. Take any $f\in C(L')$ and $g=f|L\in C(L)$.
Then
\[\nu_n(f)-\delta_n(f)=\nu_n(g)-f(n)=\varphi(n)(g)-f(n)\to 0.\]
Indeed,
if the set $N=\{n\in\omega: |\varphi(n)(g)-f(n)|\ge\eps\}$ were infinite, for some $\eps>0$,  then it would have an accumulation point $t\in L$ and
$|\varphi(t)(g)-f(t)|=|g(t)-f(t)|\ge \eps$, a contradiction.

$(iii)\rightarrow (i)$. We can extend a function $g\in C(L)$ to $Eg\in C(L')$ setting  $Eg(n)=\nu_n(g)$ for $n\in\omega$.
By (iii) $Eg$ is indeed continuous, and $E$ is a bounded operator since $\nu_n$ are bounded.
\end{proof}

The following result is essentially due to Kubi\'s, see \cite{KKL11}; recall that a space $L$ supports a strictly positive measure
if there is $\mu\in P(L)$ such that $\mu(U)>0$ for every nonempty open $U\sub L$.

\begin{theorem}\label{Kubis}

\begin{enumerate}[(a)]
\item
If $\gamma\omega$ is a tame compactification of $\omega$ then its remainder $\gamma\omega\!\sm\!\omega$ supports
a strictly positive measure.
\item
Let $K$ be a compact space if weight $\omega_1$ which does not support a measure.
Then there is a nontrivial twisted sum of $c_0$ and $C(K)$.
\end{enumerate}
\end{theorem}

\begin{proof}
The first assertion can be easily derived from Lemma \ref{p:2}, see \cite[Theorem 17.3]{KKL11} or \cite[Theorem 5.1]{DP16}.

By the Parovi\v{c}enko theorem, $K$ satisfying the assumptions from (b) is a continuous image of $\beta\omega\!\sm\!\omega$ so there is a compactification
$\gamma\omega$ which remainder is homeomorphic to $K$. It follows from  (a) that  the compactification  $\gamma\omega$ is not tame so,
by Lemma \ref{p:1},  $C(\gamma\omega)$ is a nontrivial twisted sum of  $c_0$ and $C(K)$.
\end{proof}

Recall that a compact space $K$  is an \emph{absolute retract} if $K$ is a retract of any compact space $L$ containing $K$ (equivalently, of any completely regular space $X$ containing $K$). Clearly, if $L'\in \cde{(L)}$ and $L$ is an absolute retract that, taking a retraction $r:L'\to L$, we
get a norm-one extension operator $E:C(L)\to C(L')$, where $Eg=g\circ r$.

We shall often discuss Boolean algebras and their Stone spaces, using
the classical Stone duality.
 Given an algebra $\fA$;   its Stone space (of all ultrafilters on $\fA$) is denoted by $\ult(\fA)$. If $K$ is compact and zerodimensional then
 $\clop(K)$ is the algebra of clopen subsets of $K$.

We write $M(\fA)$ for the space of all signed {\em finitely}  additive functions on an algebra $\fA$; likewise, for any $r\ge 0$, $M_r(\fA)$ denotes the family
of $\mu\in M(\fA)$ for which $\|\mu\|\le r$. Here, as usual, $\|\mu\|=|\mu|(X)$ and $|\mu|$ is the variation of $\mu$.
Recall that $M(\fA)$ may be identified with $M(\ult(\fA))$ because every $\mu\in M(\fA)$ defines via the Stone isomorphism an additive function
on $\clop(\fA)$ and such a function extends uniquely to a Radon measure.

Given any subfamily  $\mathcal{F}$ of an algebra $\fA$, we denote by $\la\mathcal{F}\ra$ the subalgebra of $\fA$ generated by $\mathcal{F}$.

\section{On twisted sums}\label{ots}

We show here that the triviality of every twisted sum of $c_0$ and $C(K)$ can be expressed, in some way, in terms of the existence
of extension operators.

\begin{definition}\label{ts:0.1}
We shall say that a compact space $K$ has the $^*$extension property if for every $L'\in\cde(M_1(K))$ there is a bounded operator
$E:C(K)\to C(L')$ such that $Eg(\nu)=\nu(g)$ for every $g\in C(K)$ and $\nu\in M_1(K)$.
\end{definition}

Let us note that $C(K)$ may be seen as a subspace of $C(M(K))$ by the usual identification of an element of a Banach space with the corresponding element of its second dual.
The operator $E$ as in Definition \ref{ts:0.1} will be called an $^*$extension operator, the one that extends $g\in C(K)$, treated as an element of $C(M(K))$,
 to a  continuous function on  $L'$.

\begin{lemma}\label{ts:0.3}
Given $K$ and $L'=K\cup\omega\in\cde(M_1(K))$, the following are equivalent

\begin{enumerate}[(i)]
\item there is an $^*$extension operator $E:C(K)\to C(L')$;
\item there is a bounded sequence $(\nu_n)_n$ in $M(K)$ such that
for every $g\in C(K)$, if $\widehat{g}\in C(L')$ is any function extending $g$, treated as a function on $M_1(K)$, then
\[\lim_n( \nu_n(g)-\widehat{g}(n))=0.\]
\end{enumerate}
\end{lemma}

\begin{proof}
$(i)\to (ii)$.
Consider an $^*$extension operator $E:C(K)\to C(L')$ and the conjugate operator $E^*:M(L')\to M(K)$. We put $\nu_n=E^*\delta_n$ for $n\in\omega\sub L'$; then
$(\nu_n)_n$ is a bounded sequence in $M(K)$.

Take any $g\in C(K)$ and its extension $\widehat{g}\in C(L')$. Then $\nu_n(g)=Eg(n)$ so
\[ \nu_n(g)-\widehat{g}(n)=Eg(n)-\widehat{g}(n)\to 0,\]
by Lemma \ref{ts:0.25}, since $Eg$ and $\widehat{g}$ are two continuous extensions of the same function, of $g$ acting on $M_1(K)$, and $L'\in \cde(M_1(K))$.

For $(ii)\to (i)$ take any $g\in C(K)$, put $Eg(\nu)=\nu(g)$, for $\nu\in M_1(K)$, and define $Eg(n)=\nu_n(g)$, for  $n\in\omega$.
Then the function $Eg$ is continuous on $L'$.
\end{proof}

We shall also need the following general fact.

\begin{lemma}\label{ts:1}
Let $T:X\to Y$ be a bounded linear surjection between Banach spaces $X$ and $Y$. Then
\[T^*[Y^*]=\ker(T)^\perp=\{x^*\in X^*: x^*|\ker(T)\equiv 0\}.\]
\end{lemma}

\begin{theorem}\label{ts:2}
If a compact space $K$ has the $^*$extension property then every twisted sum of $c_0$ and $C(K)$ is trivial.
\end{theorem}

\begin{proof}
Fix a short exact sequence
\[0\to c_0\stackrel{i}{\rightarrow} X \stackrel{T}{\rightarrow} C(K)\to 0\,.\]

Let $Z=i[c_0]$. By $e_n$ we denote the $n$-th unit vector in $c_0$ and let $e_n^*\in \ell_1=(c_0)^*$ be the corresponding dual functional.
Put $x_n=i(e_n)$ for every $n$.
Then there is a norm bounded
sequence $(x_n^*)_n$ in $X^*$ such that $i^*x_n^*=e_n^*$. Suppose that  $\|x_n^*\|\le r_0$ for every $n$.

Note that the set $\{x_n^*: n\in\omega\}$ is $ weak^*$ is discrete. Let
\[ L=T^*[M_r(K)]\sub X^*,\]
 where $r>0$ is taken big enough so that $L$ contains $\{x^*\in Z^\perp: \|x^*\|\le r_0\}$.

% Then $L$ is homeomorphic to $M_1(K)$ since $T^*$ is injective.
We now consider  $L'=L\cup\{x_n^*: n\in\omega\}$ equipped with the $weak^*$ topology.
\medskip

\noindent {\sc Claim.} $L'$ is a countable discrete extension of $L$.
\medskip

Indeed, it is enough to notice that if $x^*$ is an accumulation point of $\{x_n^*: n\in\omega\}$ then $x^*\in Z^\perp$ but
this follows from the fact that for $n>k$
\[ x^*_n(i(e_k))=i^*x_n^*(e_k)=e_n^*(e_k)=0.\]

Let $L''=M_1(K)\cup\omega$. We define a topology on the set $L''$ in the following way:

Consider the mapping
\[ h: L''\to L'=T^*[M_r(K)]\cup \{x_n^*: n\in\omega\},\]
defined by $h(\nu)=T^*(r\nu)$ for $\nu\in M_1(K)$ and $h(n)=x_n^*$ for $n\in \omega$. Then $h$ is a bijection since
$T^*$ is injective and $x_n^*\neq x_k^*$ for $n\neq k$. We topologize $L''$ so that $h$ becomes a homeomorphism; clearly
$M_1(K)$ gets its usual $weak^*$ topology when treated as a subspace of $L''$.

Since $K$ has the $^*$extension property, by Lemma \ref{ts:0.3} applied for $L''$, there is a bounded sequence $(\nu_n)_n$ in $M(K)$ satisfying
\ref{ts:0.3}(ii). Let $z_n^*=T^*(r\nu_n)$ for $n\in\omega$. Then $(z_n^*)_n$ is a bounded sequence in $X^*$ and the following holds.
\medskip

\noindent {\sc Claim.} $z_n^*-x_n^*\to 0$ in the $weak^*$ topology of $X^*$.
\medskip

Take any $x\in X$; then (treating $x$ as an element of $X^{**}$), $x\circ h\in C(L'')$ and
\[ x\circ h(\nu) = T^*(r\nu)(x) = \nu(rTx),\]
for $\nu\in M_1(K)$.
This means that $x\circ h$ is an extension of $rTx\in C(K)$ treated as a function on $M_1(K)$. Therefore
\[ z_n^*(x)-x_n^*(x)=r\nu_n(Tx)-h(n)(x) = \nu_n(x\circ h) - x\circ h(n) \to 0,\]
as required.
\medskip

Define now
\[ P:X\to X,\quad Px=\sum_n \left(x_n^*(x)-z_n^*(x)\right)\cdot x_n.\]
Note that $Px_k=x_k$ since $x_n^*(x_k)=1$ if $n=k$ and is 0 otherwise; moreover,
$z_n^*(x_k)=0$ for every $n$ and $k$.
Using above Claim,  we conclude that $P$ is indeed a projection onto $Z$, and the proof is complete.
\end{proof}

\begin{remark}\label{ts:3}
Using the construction from the above proof we can show that if $X$ is a twisted sum of $c_0$ and $C(K)$ then $X$ is isomorphic to a subspace of
$C(L')$, where $L'$ is a countable discrete extension of $M_1(K)$.

%{\color{blue} We should clarify this and prove that every twisted sum $X$ of $C(K)$ and $c_0$ is equivalent, in the `arrow' language of Castillo to
%$X_1$, where $X_1$ is the image of an operator sending $x\in X$ to $x$ treated is a continuous function on $L'$, $L'\sub X^*$ is as in the proof of \ref{ts:2}.}
\end{remark}

\section{Asymptotic measures on Boolean algebras}\label{am}

We consider here an algebra $\fA$ of subsets of some set $X$. In the sequel, $\fB$ (with possible indexes) always denotes a finite subalgebra of $\fA$.
We introduced in Section \ref{pr} the symbol $M(\fA)$ denoting the space of all signed finitely additive functions $\fA\to\er$ of bounded variation.

Given {\em any} real-valued partial functions $\vf,\psi$ on $\fA$  and an algebra $\fB$ contained in their domains we write
\[ \dist_\fB (\vf,\psi)=\sup_{B\in\fB} |\vf(B)-\psi(B)|.\]

The aim of this section is to analyze a certain property of Boolean algebras that is, as it is explained in the next section,
closely related to the (non)existence of nontrivial twisted sums
of $c_0$ and $C(K)$ spaces.

\begin{definition}\label{am:-1}
We shall say  that a Boolean algebra $\fA$ has the {\em approximation property} if for any sequence of  {arbitrary} functions $\vf_n:\fA\to [-1,1]$
satisfying
\[\inf_{\mu\in M_1(\fA)} \dist_\fB(\mu, \vf_n) \to 0,\]
for every finite algebra $\fB\sub\fA$,  there is a  bounded sequence $(\nu_n)$ in $M(\fA)$ such that
\[ \vf_n(a)-\nu_n(a)\to 0 \mbox{  for all }a\in \fA.\]
\end{definition}

Note that the condition on the asymptotic behavior of $\vf_n$ above means that $\vf_n$ is {\em asymptotically} a measure of norm $\le 1$ on a given finite algebra  $\fB$ (it satisfies the corresponding conditions up to some small constant whenever $n$ is large enough). By a standard diagonal argument one can check that every countable algebra $\fA$ has the approximation property.
We shall prove below that some uncountable algebras may also have such a property.

\begin{notation} \label{am:0}
We fix now  a sequence $(\vf_n)_n$ of arbitrary  set functions $\vf_n:\fA\to [-1,1]$ on an algebra $\fA$.
We shall  write $\mq(\fB)$ (or $\mq_r(\fB)$) for
the set of signed measures having rational values (and having the norm $\le r$, respectively).

For any $\fB$ and  $n$ we define
\[o_n(\fB)=\inf\{ \dist_\fB(\nu,\vf_n): \nu\in \mq_1(\fB)\}. \]
\end{notation}

Of course, to calculate the value of  $o_n(\fB)$ we might as well replace $\mq_1(\fB)$ by $M_1(\fB)$ (as in Definition \ref{am:-1}). It will be convenient, however, to
consider in the sequel measures on finite algebras having only rational values.

For the sake of the proof of Theorem \ref{am:6} below we need to introduce some additional parameters.
Given two subalgebras $\fB_1,\fB_2\sub\fA$ and $\nu_i\in M(\fB_i)$ for $i=1,2$, we say that $\nu_1,\nu_2$ are consistent if
$\nu_1(B)=\nu_2(B)$ for every $B\in\fB_1\cap \fB_2$.

\begin{definition}\label{am:1}
Fix $r> 1$.
We define $O_n(\fB)$ (positive real or $+\infty$), for $n\in\omega$ and finite $\fB\sub\fA$, by induction on $|\fB|$.
Set $O_n(\fC)=1/(n+1)$ for every $n$ in the case of the trivial algebra $\fC$. Suppose that $O_n(\fC)$ has been defined for
every proper subalgebra $\fC$ of $\fB$.
Then we put
\[O_n(\fB)=C_0+o_n(\fB)+ 1/(n+1),\]
where $C_0$ is the infimum  of $C>0$ such that

\begin{enumerate}[(i)]
\item whenever $\fB_1,\fB_2\sub \fB$ are proper subalgebras and we are given two consistent
measures  $\nu_i\in \mq_r(\fB_i)$ that  satisfy  $\dist_{\fB_i} (\nu_i,\vf_n) < O_n (\fB_i)$ for $i=1,2$,
then there is $\mu\in \mq_r(\fB)$ such that $\mu$ is a common extension of $\nu_1$ and $\nu_2$ and
$\dist_\fB(\mu,\vf_n)\le C$;

\item for any proper subalgebra $\fC\sub\fB$ and a measure $\nu\in \mq_r(\fC)$ with  $\dist_{\fC} (\nu,\vf_n) < O_n(\fC)$ there is
 an extension of $\nu$ to $\mu\in \mq(\fB)$ such that $\|\mu\|\le\max(\|\nu\|,1)$ and  $\dist_{\fB} (\mu,\vf_n) \le  C$.
\end{enumerate}
\end{definition}

\begin{remark}\label{am:2}
The definition of $O_n$ depends on the chosen parameter $r$; we write $O_n$ rather than $O^r_n$ for simplicity.
Obviously, $o_n(\fB)\le O_n(\fB)$ for all $n$ and $\fB$.

Note that in case of $(ii)$ above the set of such $\mu$ is always nonempty, see Lemma \ref{ap:2} from Appendix at the end of the paper.
However, there may be no common extension of $\nu_1,\nu_2$ considered in $(i)$ which would satisfy $\|\mu\|\le r$;
in such a case we understand that  $O_n(\fB)=+\infty$.
\end{remark}

\begin{lemma}\label{am:3}
If $\lim_n o_n(\fB)=0$ for every finite algebra $\fB\sub\fA$ then $\lim_n O_n(\fB)=0$ for every such $\fB$.
\end{lemma}

\begin{proof}
We argue by induction  on $|\fB|$. If $\fB$ is trivial then $O_n(\fB)=1/(n+1)$. Suppose that $\fB$ is nontrivial and $\lim_n O_n(\fC)=0$ for any proper subalgebra $\fC$ of $\fB$.

Let $N\ge 2$ be the number of atoms of $\fB$. Fix  $\eps>0$ and take $\delta>0$ such that
\[ 4N\delta<r-1 \mbox{ and } (4N+1)\delta<\eps.\]
By the inductive assumption and the fact that  $\lim_n o_n(\fB)=0$ there is $n_0$ such that for $n\ge n_0$ we have $o_n(\fB)<\delta$ and
$O_n(\fC)<\delta$ for all proper subalgebras $\fC$ of $\fB$. We shall check that then $O_n(\fB)\le\eps$ whenever $n\ge n_0$. Given such $n$, we will verify that $C=\eps$ satisfies conditions (i-ii) of Definition \ref{am:1}.

Consider a pair $\fB_1,\fB_2\sub \fB$ of proper subalgebras and a pair $\nu_i\in\mq_r(\fB_i)$ of consistent measures as in
Definition \ref{am:1}(i).
Take $\lambda\in \mq_1(\fB)$ witnessing $o_n(\fB)<\delta$.
Then $\dist_{\fB_i}(\nu_i,\lambda)<2\delta$ so by Lemma \ref{ap:1} there is a common extension of $\nu_1,\nu_2$ to a measure
$\lambda'\in \mq(\fB)$ such that $\|\lambda-\lambda'\|<4N\delta$. This implies
\[ \|\lambda'\|\le \|\lambda\|+ 4N\delta\le 1 +r-1=r;\]
\[ \dist_\fB(\lambda',\vf_n)\le \dist_\fB(\lambda',\lambda)+\dist_\fB(\lambda, \vf_n)< 4N\delta +\delta< \eps,\]
as required.

Consider now $\fC$ and $\nu\in\mq_r(\fC)$ as in  Definition \ref{am:1}(ii).

Let, again,   $\lambda\in \mq_1(\fB)$ witnesses that $o_n(\fB)<\delta$.
We have $\dist_\fC(\nu,\lambda)<2\delta$ so
Lemma \ref{ap:2} gives us a measure $\lambda'\in\mq(\fB)$ extending $\nu$ with $\|\lambda'\|\le\max(\|\nu\|, 1)$ and such that
$\dist_\fB(\lambda',\lambda)<6\delta$. It follows that  $\dist_\fB(\lambda',\vf_n)<7\delta<(4N+1)\delta <\eps$, as required.
\end{proof}

%As we  mentioned in Remark \ref{am:2},
%Given two subalgebras $\fB_1,\fB_2\sub\fA$ and $\nu_i\in M(\fB_i)$ for $i=1,2$, we say that $\nu_1,\nu_2$ are consistent if
%$\nu_1(B)=\nu_2(B)$ for every $B\in\fB_1\cap \fB_2$. It is easy to check that a pair of consistent measures admits a common extension to
%a measure defined on an algebra generated by $\fB_1\cup \fB_2$; however, as it is  explained in the appendix, the norm of such a common extension may be quite big.

As we mentioned in Remark \ref{am:2}, there is a problem of controlling the norm of a common extension of two consistent measures.
To handle this we need to introduce another property of a given algebra $\fA$.

\begin{definition}\label{am:5}
Let us say that a Boolean algebra $\fA$ has $\lep(r)$ (local extension property) for some $r > 1$ if there is a family $\bB$ of finite subalgebras of $\fA$
such that

\begin{enumerate}[(i)]
\item for every finite algebra $\fC\sub\fA$ there is $\fB\in\bB$ with $\fB\supseteq \fC$;
\item whenever $\bB'\sub \bB$ is uncountable then there are distinct $\fB_1,\fB_2\in\bB'$ such that
any pair of consistent measures $\nu_i\in \mq_1(\fB_i)$ admits a common extension to a measure $\nu\in \mq_r(\langle \fB_1\cup\fB_2\rangle )$.
\end{enumerate}
\end{definition}

We are now ready for the main result of this section which requires Martin's axiom.  For the reader's convenience we recall briefly
the statement of the axiom in the form that we shall use. We essentially follow here the terminology of Fremlin \cite{Fr84};
in particular, in a partially ordered set all the properties mentioned here are defined `upwards'.
Let $(\bP, \leq)$ be a partially ordered set; elements  of $\bP$ are often called  conditions.

A set $\bD\sub\bP$ is {\em dense} if for every $p\in\bP$ there is $d\in\bD$ with
$p\leq d$. A set $\bG\sub \bP$ is {\em directed} if for every $p_1,p_2\in\bG$ there is $p_3\in\bG $ such that $p_1,p_2\leq p_3$.

A partially ordered set $\bP$ is said to satisfy $ccc$ (the countable chain condition) if every uncountable $P\sub \bP$ contains two distinct
conditions $p_1,p_2$ that are comparable, i.e.\ there is $q\in\bP$ such that $p_1,p_2\leq q$.

For a given cardinal number $\kappa$,  $\MA(\kappa)$ reads as follows:

\begin{axiom}
For every $ccc$ partially ordered set $\bP$ and a family $\{\bD_\xi: \xi<\kappa\}$ of its dense subsets there is a directed
set $\bG\sub\bP$ such that $\bG\cap \bD_\xi\neq\emptyset$ for every $\xi<\kappa$.
\end{axiom}

Recall that $\MA(\omega)$ is true and $\MA(\con)$ is false; the main  use of Martin's axiom stems from the fact that
$\MA(\omega_1)$ is relatively consistent, see \cite{Fr84} and references therein.
Recall also that $\MA$ is the commonly used  abbreviation  of the statement  `$\MA(\kappa)$ for every $\kappa<\con$'.

\begin{theorem}\label{am:6}
Suppose that $\fA$ is an algebra with $|\fA|=\kappa$. Suppose further  that $\fA$ has $\lep(r)$ for some $r> 1$ and that
$\ult(\fA)$ is separable.

Then, assuming     $\MA(\kappa)$, the algebra $\fA$ has the approximation property.
\end{theorem}

\begin{proof}
Let  us fix a sequence
$\vf_n:\fA\to [-1,1]$  of functions  as in Definition \ref{am:-1}
Our task is to construct a bounded sequence $(\nu_n)$ in $M(\fA)$ satisfying the assertion
of \ref{am:-1}. Note that $\lim_n o_n(\fB)=0$ for every finite subalgebra $\fB$ of $\fA$, see Notation \ref{am:0} and the remark after it.
We use below the parameters $O_n$ introduced in Definition \ref{am:1} (for $r$ from the assumption of the theorem).
Let $\bB$ be a family of subalgebras granted by $\lep(r)$.

We consider a partially ordered set $\bP$ of conditions
\[p=(\fB, n, (\nu_i)_{i \le n},k),\mbox{ where }\]
\begin{enumerate}[(i)]
\item $\fB\in\bB$  and $n,k$ are positive integers;
\item for every $i\le n$, the measure $\nu_i$ is in $\mq_r(\fB)$;
\item $\dist_\fB(\nu_i,\vf_i) < O_i(\fB)$ for any $i\le n$;
\item $O_m(\fB)<1/k$ for every $m\ge n$.
\end{enumerate}

Consider two conditions
\[p=(\fB,n, (\nu_i)_{i\le n},k), \quad p'=(\fB',n', (\nu_i')_{i\le n'},k')\in\bP.\]
We shall say that $p'$ is a simple extension of $p$ if $k'\ge k$ and

\begin{enumerate}[---]
\item either $\fB'=\fB$ and  $n'\ge n$, $\nu_i'=\nu_i$ for $i\le n$,
\item or $n'=n$, $\fB\sub \fB'$ and $\nu_i'$ extends $\nu_i$ for every $i\le n$.
\end{enumerate}

Then we define a partial order on $\bP$ declaring $p\leq p'$ if there are $s\in\omega$ and a sequence $p_j$, $j=0,\ldots, s$,  in $\bP$ such that
$p_0=p$, $p_s=p'$ and $p_{j+1}$ is a simple extension of $p_j$ for every $j<s$. Note that $\leq$ is indeed a partial order on $\bP$.
\medskip

\noindent{\sc Claim A.}
Let $A\in\fB$ and  let $p,p'\in\bP $ be specified as above. If    $p\leq p'$
then
\[ |\nu_i'(A) -\vf_i(A)|< 1/k  \mbox{ whenever } n\le i\le n'.\]
%\medskip

Note that  if $p'$ is a simple extension of $p$ with $\fB'=\fB$ then for $i\ge n$ we have
\[ |\nu_i'(A) -\vf_i(A)|<O_i(\fB')=O_i(\fB)<1/k,\]
by $(iv)$. If $p'$ is a simple extension of $p$ with $n'=n$ then the inequality holds
as $\nu_i'(A)=\nu_i(A)$ for $i\le n$.
Hence the assertion follows by induction on the number of simple extensions
leading from $p$ to $p'$.
\medskip

%Note that
%$p\leq p'$ means that the condition $p'$ is stronger; accordingly,  we consider $ccc$ and other properties to be defined upwards.
%\medskip

\noindent{\sc Claim B.} $\bP$ is $ccc$.
\medskip

Consider an uncountable family $P\sub\bP$ of conditions
\[ p=(\fB^p,n^p, (\nu_i^p)_{i \le n^p}, k^p).\]
Shrinking $P$ if necessary, we can assume that $n^p=n$ and  $k^p=k$ are constant for $p\in P$.

Let $S$ be a countable dense subset of $\ult(\fA)$. Every  $x\in S$ defines a 0-1 probability measure $\delta_x\in M(\fA)$, where $\delta_x(B)=1$ iff $B\in x$.
Let $M^S$ be the countable family of all measures on $\fA$ that are rational linear combinations of $\delta_x$'s with $x\in S$.
Note that any measure $\nu\in\mq(\fB)$ on a finite algebra $\fB$ can be represented as a restriction of some $\widetilde{\nu}\in M^S$ to $\fB$, where
$\|\widetilde{\nu}\|=\|\nu\|$.

Using the above remark, thinning $P$ out again, we can assume that for every
$i\le n$ there is $\widetilde{\nu}_i\in M^S$  such that $\nu^p_i=\widetilde{\nu}_i|\fB^p$ and
 $\|\nu^p_i\|=\|\widetilde{\nu}_i|\fB^p\|$ for every $p\in P$.

Finally, we apply $\lep(r)$ to choose  distinct  $p,q\in P$ so that $\fB^p$ and $\fB^q$ have the property granted by Definition \ref{am:5}(ii).
We put  $\fB_0=\fB_1\cap \fB_2$ and, using \ref{am:5}(i)  choose $\fB\in\bB$ containing $\fB_1\cup\fB_2$.
By Lemma \ref{am:3} there is  $n_1\ge n$ such that $O_m(\fB)<1/k$ for every $m\ge n_1$.
We shall check that $p$ and $q$ have a common extension in $\bP$.

For every $i$ such that $n< i\le n_1$ we choose a measure $\pi_i\in \mq_1(\fB_0)$    such that
\[ \dist_{\fB_0}(\pi_i,\vf_i)<o_i(\fB_0)+1/(i+1)\le O_i(\fB_0),\]
and then by part (ii) of Definition \ref{am:1} extend $\pi_i$ to measures
\[ \nu_i^p\in \mq_1(\fB^p) \mbox{ and } \nu_i^q\in \mq_1(\fB^q) \mbox{ such that }\]
\[ \dist_{\fB^p}(\nu_i^p,\vf_i)<O_i(\fB^p),
 \dist_{\fB^q}(\nu_i^q,\vf_i)<O_i(\fB^q).\]
Then there is $\nu_i$ in $\mq_r(\fB)$ which is a common extension of $\nu_i^p$ and $\nu_i^q$ and such that
$\dist_{\fB}(\nu_i,\vf_i)< O_i(\fB)$; indeed, if $O_i(\fB)<+\infty$, then this follows from  Definition \ref{am:1}(i). In case
 $O_i(\fB)=+\infty$ we may take {\em any} extension granted by \ref{am:5}(ii) and the way we have chosen $\fB^p$ and $\fB^q$, and extend it to $\fB$ preserving its norm.

For $i\le n$ we choose $\nu_i\in \mq_r(\fB)$ applying Definition \ref{am:1} to the pair $\nu_i^p,\nu_i^q$. Note that if $O_i(\fB)=+\infty$ then
we may use the fact that both $\nu_i^p$ and $\nu_i^q$ are represented by the same measure $\widetilde{\nu}_i\in M^S$ so
$\widetilde{\nu}_i|\fB$ is their common extension to $\fB$ with norm $\le r$.

In this way we get  simple extensions
\[ p_1=(\fB^p, n_1, (\nu_i^p)_{i\le n_1},k), \quad q_1=(\fB^q, n_1, (\nu_i^q)_{i\le n_1},k),\]
of $p$ and $q$, respectively. In turn,
\[ s=(\fB, n_1, (\nu_i)_{i\le n_1}, k)\in \bP\]
satisfies  $p_1,q_1\leq s$, and this finishes the proof of Claim B.
 \medskip

\noindent{\sc Claim C.} For every $k_0,n_0\in\omega$  and finite $\fB_0\sub \fA$,
the set
\[\bD(\fB_0,n_0,k_0)=\{p=(\fB^p,n^p, (\nu_i^p)_{i \le n^p}, k^p)\in\bP: \fB^p\supseteq \fB_0, n^p\ge n_0, k^p\ge k_0\},\]
%\[\bD(k_0)=\{p=(F^p,n^p, (\nu_i^p)_{i < n^p}, k^p)\in\bP:  k\ge k_0\}\]
is dense in $\bP$.
\medskip

 Take any $p=(\fB^p,n^p, (\nu_i^p)_{i \le n^p}),k^p) \in\bP$
 and consider a triple $\fB_0$, $n_0, k_0$; we can assume that $k_0\ge k^p$ and $n_0\ge n^p$.

  Find $\fB\in\bB$ containing $\fB^p\cup\fB_0$
and $n_1\ge n_0$ such that $O_m(\fB)<1/k_0$ for $m\ge n_1$.
 Then, arguing as in the proof of Claim B we define appropriate $\nu_i$ and $\nu_i'$ so that
\[p\leq (\fB^p,n_1, (\nu_i)_{i\le n_1}, k_0)\leq (\fB, n_1, (\nu_i')_{i\le n_1}, k_0).\]
Indeed, for $n^p<i\le n_1$ we pick a measure $\nu_i\in \mq_1(\fB^p)$  such that
\[ \dist_{\fB^p}(\nu_i,\vf_i)<o_i(\fB^p)+1/(i+1)\le O_i(\fB^p),\]
and then by part (ii) of Definition \ref{am:1} extend $\nu_i$ to a measure
$\nu_i'\in \mq_1(\fB)$  such that $ \dist_{\fB}(\nu_i',\vf_i)<O_i(\fB)$.
Accordingly, for $i\le n$ we suitably extend every $\nu_i$ to $\nu_i'\in \mq_r(\fB)$.
\medskip

With Claim B and C at hand, we apply Martin's axiom $\MA(\kappa)$ to get a directed set $\bG\sub \bP$ such that
$\bG\cap \bD(\fB_0,n_0,k_0)\neq\emptyset$ for every finite $\fB_0\sub\fA$ and positive integers  $n_0,k_0$. 
Consider conditions
\[p=(\fB^p,n^p, (\nu_i^p)_{i \le n^p}, k^p)\in \bG.\]
Given $i\in\omega$, the family $\bG_i=\{p\in\bG: i\le n^p\}$ defines a family $\{\nu_i^p: p\in\bG_i\}$
 of consistent measures $\nu_i$ of variation $\le r$ and their domains cover all of $\fA$. Therefore   all those measures extend uniquely to a measure
$\mu_i\in M_r(\fA)$.

For any $A\in\fA$ there is  $\fB_0\in\bB$ such that $A\in\fB_0$. Given $\eps>0$, take $k_0$ such that $1/k_0<\eps$ and
\[ p=(\fB,n,(\nu_i)_{i\le n},k) \in\bG\cap \bD(\fB_0,1,k_0).\]
 Then $\mu_n(A)=\nu_n(A)$
so
\[|\mu_n(A)-\vf_n(A)| < O_n(\fB)<1/k\le 1/k_0<\eps.\]
For every $m>n$ there is $p'=(\fB',n',(\nu_i')_{i\le n'}, k')\in\bG$ such that  $p\leq p'$ and $m\le n'$. Then $\mu_m(A)=\nu_m'(A)$
and $|\nu_m(A)-\vf_m(A)|<1/k_0$ by Claim A.
This shows that
\[\mu_n(A)-\vf_n(A)\to 0,\]
and the proof is complete.
\end{proof}

\begin{proposition}\label{am:7}
For any cardinal number $\kappa$, the algebra $\fA=\clop(2^\kappa)$ has $\lep(2)$.

Consequently, if $\MA(\kappa)$ holds for some $\kappa<\con$ then the algebra $\clop(2^\kappa)$
has the approximation property.
\end{proposition}

\begin{proof}
For any finite set $F\sub\kappa$ we let $\fB_F$ be the finite subalgebra of  $\fA$  of all sets that are determined by  coordinates in $F$;
thus $\fB_F$ is generated by its atoms $A$ of the form
\[A=\{t\in 2^{\kappa}: t|F=\tau|F\},\]
for some function $\tau:F\to 2$.
Clearly the family $\bB$ of all  such $\fB_F$ is cofinal in $\fA$ so it is enough to check that any $\fB_{F_1}, \fB_{F_2}$
satisfy (ii) of Definition \ref{am:5}.

Consider  $\nu_i\in \mq_1(\fA_{F_i})$, $i=1,2$ and suppose that  $\nu_1$ and $\nu_2$ agree on
$\fB_{F_1}\cap \fB_{F_2}$ which is $\fB_H$, where $H=F_1\cap F_2$.

Let $A_i$, $i\le 2^{F_1\sm H}$, be the list of all atoms of $\fB_{F_1\sm H}$ and, accordingly,
 $B_j$ be the list of all atoms of $\fB_{F_2\sm H}$.

 For a fixed atom $C$ of $\fA_{H}$ we apply Lemma \ref{ap:3} to $a_i=\nu_1(A_i\cap C)$ and $b_j=\nu_2(B_j\cap C)$.
 Note that
 \[\sum_i a_i=\nu_1(C)=\nu_2(C)=\sum_j b_j.\]
 This enables us to define $\overline{\nu}$ for all  $A\in\fA_F$ contained in $C$, and
 \[|\overline{\nu}|(C)\le  \max\big(|\nu_1|(C), |\nu_2|(C)\big)\le    |\nu_1|(C)+|\nu_2|(C),\]
 so we get $\overline{\nu}$ with  $\| \overline{\nu}\|\le 2$.

The final statement follow directly from Theorem \ref{am:6}.
\end{proof}

\begin{remark}\label{am:7.5}
In the proof of \ref{am:7} one can alternatively check that $\nu_1$ and $\nu_2$ admit a common extension of norm $\le 2$
applying  a result Basile, Rao and Shortt \cite{BRS94} described in Appendix.
One can check that $SC(\nu_1,\nu_2)\le 2$ basing on the following remark: If $B_i\in \fB_{F_i}$ and $B_1\sub B_2$ then
there is $C\in \fB_{F_1\cap F_2}$ such that $B_1\sub C\sub B_2$.
\end{remark}

Recall that a family $\cA$ of subsets of $\omega$ is {\em almost disjoint} if the intersection of any two distinct members of $\cA$ is finite.

\begin{proposition}\label{am:8}
Let $\fA$ be an algebra of subsets of $\omega$ that is generated by an almost disjoint family $\cA$ and all finite subsets of $\omega$.
Then the algebra  $\fA$ has $\lep(3)$.

Consequently, if $|\fA|=\kappa<\con$ and $\MA(\kappa)$ holds  then the algebra $\fA$
has the approximation property.
\end{proposition}

\begin{proof}
We consider the family $\bB$ of finite subalgebras of $\fA$, where every $\fB=\la n, A_1,\ldots, A_k\ra\in\bB$  is
an algebra generated  by all subsets of $n=\{0,1,\ldots, n-1\}$ and $A_i\in\cA$ having the property that
$A_i\cap A_j\sub n$ for $i\neq j$. Clearly $\bB$ is cofinal in $\fA$; we shall check that  (ii) of Definition \ref{am:5} holds for $r=3$.

If $\bB'\sub\bB$ is uncountable then there are two algebras in $\bB'$ of the form
\[\fB_1=\la n,A_1,\ldots, A_m, B_1,\ldots, B_k\ra, \quad \fB_2=\la n,A_1,\ldots, A_m, C_1,\ldots, C_k\ra,\]
where $B_i, C_j\in \cA$ are all distinct. Set
\[ X=n\cup\bigcup_{i\le m} A_i \cup \bigcup_{i,j\le k} B_i\cap C_j;\]
note that  $B_i\sm X$ and $C_i\sm X$ are infinite for $i\le k$.

Take $\nu_1\in \mq_1(\fB_1)$ and $\nu_2\in \mq_1(\fB_2)$ which agree on $\fB_1\cap\fB_2=\la n,A_1,\ldots, A_m\ra$. We can represent them as the restrictions of
\[ \nu_1=\nu_0+\sum_{i\le k} b_i\delta_{x_i}+b \delta_x,\quad \nu_2=\nu_0+\sum_{i\le k} c_i\delta_{y_i}+c \delta_x, \mbox{ where}\]
\[ x_i\in B_i\sm X,  y_i\in C_i\sm X,  x \in \omega\sm \big(X\cup\bigcup_{i\le k} (B_i\cup C_i)\big).\]
Here $\nu_0$ is defined as $\nu_0(A)=\nu_1(A\cap X)=\nu_2(A\cap X)$.
Write $\overline{b}=\sum_{i\le k} b_i$ and $\overline{c}=\sum_{i\le k} c_i$, and  consider the measure
\[\nu=\nu_0+\sum_{i\le k} b_i\delta_{x_i}+ \sum_{i\le k} c_i\delta_{y_i}+ (b-\overline{c})\delta_x.\]
Then we have
\[ \nu(\omega)=\nu_0(X)+\overline{b}+\overline{c} +b - \overline{c}=\nu_1(\omega)=\nu_2(\omega).\]
Moreover, $\nu(B_i)=\nu_0(B_i\cap n)+b_i=\nu_1(B_i)$, a similar argument holds for $\nu_2$ and $C_i$. It follows that
$\nu$ is a common extension of $\nu_1,\nu_2$. Clearly, $\|\nu\|\le 3$, so this finishes the proof of the first assertion;
the second one follows again from Theorem \ref{am:6}.
\end{proof}

In connection with results given in the next section it would be desirable to answer the following question.

\begin{problem}
Is there a ZFC example of an uncountable algebra with the approximation  property?
\end{problem}

This problem may have a negative answer; one might conjecture, for instance, that no algebra of cardinality $\ge \con$ has the  approximation property.

\section{Trivial twisted sums of $c_0$ and $C(K)$}\label{tts}

We conclude our considerations from previous section and show here, in particular,  that under Martin's axiom and the negation of the continuum hypothesis every twisted sum of $c_0$ and $C(2^{\omega_1})$ is trivial.
Another example of this kind is a space defined from an almost disjoint family of small size.

\begin{theorem}\label{tts:1}
Let $K$ be a zerodimensional compact space  such that $\fA=\clop(K)$ has the approximation property.
Then  $K$ has the $^*$extension property.
\end{theorem}

\begin{proof}
Set  $L=M_1(K)$; fix $L'\in\cde(L)$ and identify $L'\sm L$ with $\omega$.

For every $A\in\fA$, the function $M_1(K)\ni\nu\to \nu(A)$ is continuous on $M_1(K)$. Denote by $\theta_A$ some its extension
to a continuous function  $L'\to [-1,1]$. Define  set functions $\vf_n$ on $\fA$ as $\vf_n(A)=\theta_A(n)$ for every $n$ and $A\in\fA$.

Recall that $o_n(\fB)$ is defined in \ref{am:0}. We have
$ \lim_n o_n(\fB)=0$,
for every finite subalgebra $\fB$ of $\fA$ by Lemma \ref{c:3} applied to the finite family $\{\theta_B: B\in\fB\}$.
Now the approximation property guarantees   that there are $r>0$ and  a sequence $\mu_n\in M_r(\fA)$ such that
\[\lim_n (\theta_A(n)-\mu_n(A))=\lim_n\big(\vf_n(A) - \mu_n(A)\big)=0,\]
for every $A\in\fA$. Every  measure $\mu_n$ extends uniquely to a Radon measure on $K$; we denote its extension
by the same symbol.

Consider the family $\GG$ of those $g\in C(K)$ such that whenever $\widehat{g}\in C(L')$ extends $g$ as a function on
$M_1(K)$ then $\mu_n(g)-\widehat{g}(n)\to 0$. As we have seen, for every $A\in\fA$, $\chi_A\in \GG$, see Lemma \ref{ts:0.25}. Using the same lemma we can easily verify that
$\GG$ is closed under finite linear combinations, hence every simple continuous function is in $\GG$.

Note that if $g,h\in C(K)$ and $\|g-h\|<\eps$ for some $\eps>0$ then, taking any extensions $\widehat{g}, \widehat{h}$ of $g$ and $h$, respectively,
we have $|\widehat{g}(n)-\widehat{h}(n)|<\eps$ for almost all $n\in\omega$. This remark implies that the family $\GG$ is closed under
 uniform limits and we hence $\GG=C(K)$.
Consequently, by Lemma \ref{ts:0.3},  $K$ has the $^*$extension property, and this finishes the proof.
\end{proof}

Now Propositions \ref{am:7} and \ref{am:8}, together with Theorem \ref{tts:1} yield the following.

\begin{corollary}\label{tts:2}
Assume that $\kappa<\con$ is such a cardinal number that $\MA(\kappa)$ holds. Then $2^{\kappa}$ has the $^*$extension property and hence, by Theorem \ref{ts:2}, every twisted sum of $c_0$ and $C(2^{\kappa})$ is trivial.
\end{corollary}

\begin{corollary}\label{tts:3}
Assume that $\kappa<\con$ and $\MA(\kappa)$ holds.
Let $\cA$ be an almost disjoint family of subsets of $\omega$ with $|\cA|=\kappa$. Let $K=\ult(\fA)$ where $\fA$ is an algebra
of subsets of $\omega$ generated by $\cA$ and all finite sets.

Then $K$ has the $^*$extension property and hence, by Theorem \ref{ts:2}, every twisted sum of $c_0$ and $C(K)$ is trivial.
\end{corollary}

The above results seem to give the first (consistent) examples of a nonmetrizable compact space $K$ for which every twisted sum of $c_0$ and $C(K)$ is trivial.
 Correa and Tausk \cite{CT16} proved, in particular,  that $C(2^\con)$ admits a nontrivial twisted sum with $c_0$. Hence the question
 about nontrivial twisted sums  of $c_0$ and $C(2^{\omega_1})$ cannot be decided within the usual set theory. This is also the case for compact spaces $K$ as in Corollary \ref{tts:3}, since Castillo proved that assuming $\CH$, for such spaces $K$, there exists a nontrivial twisted sum of $c_0$ and $C(K)$, see Theorem \ref{Castillo}.

The problem arises, if we can apply the above argument to any separable compactum of weight $<\con$.
In other words, we do not know if every small Boolean algebra having a separable Stone space has, consistently, the approximation property.

\begin{problem}\label{tts:4}
 Is there a ZFC example of a  separable compact space $K$ of weight $\omega_1$ such that $c_0$ and $C(K)$ have a nontrivial twisted sum?
 \end{problem}

Let us note  that if we could, while examining the $^*$extension property of a compactum $K$,  exchange $M_1(K)$ for $P(K)$,
the space of probability measures on $K$, then the way to Corollary \ref{tts:2}  would be much shorter, at least
for $\kappa=\omega_1$.
Indeed, $P(2^{\omega_1})$ is homeomorphic to $[0,1]^{\omega_1}$ and therefore $P(2^{\omega_1})$ is an absolute retract,
In particular, for every $L'\in\cde( P(2^{\omega_1}))$ there is a retraction $L'\to P(2^{\omega_1})$ so there is
a norm-one extension operator $C(P(2^{\omega_1}))\to C(L')$.
However, we prove in the next section that $M_1(2^{\omega_1})$ is not an absolute retract.
Note that $M_1(2^{\omega_1})$ is clearly a dyadic space but this fact itself does not help as the examples given in Section \ref{dyadic} indicate.

\section{On properties of $M_1(K)$}\label{mone}

Recall that a compact space $K$ is a \emph{Dugundji space} if for every compact space $L$ containing $K$ there exists a \emph{regular} extension operator $E: C(K)\to C(L)$, i.e.\ an extension operator of  norm $1$ preserving constant functions. It is well-known that a convex compact space $K$ is a Dugundji space if and only if it is an absolute retract, cf. \cite[Sec. 2]{Ha74}. For a nonmetrizable compact space $K$, the space $P(K)$ can be an absolute retract, namely Ditor and Haydon \cite{DH76} proved that $P(K)$ has this property if and only if $K$ is a  Dugundji space of weight at most $\omega_1$. We will show that this can never happen for the space $M_1(K)$.

\begin{theorem}\label{not_AR}
If $K$ is a nonmetrizable compact space, then the space $M_1(K)$ is not a Dugundji space, in particular, it is not an absolute retract.
\end{theorem}

We will prove this theorem using spectral theorem of Shchepin, the key ingredient will be Proposition \ref{not_open} below.

For a surjection $\varphi: L\to K$ between compact spaces $K,L$, $\varphi^*: M_1(L)\to M_1(K)$ denotes the canonical surjection associated with $\varphi$, i.e., the surjection given by the operator conjugate to the isometrical embedding of $C(K)$ into $C(L)$ induced by $\varphi$.
In other words, for $\mu\in M_1(L)$, $\vf^*(\mu)\in M_1(K)$ is defined by $\vf^*(\mu)(B)=\mu(f^{-1}[B])$ for Borel sets $B\sub K$.

\begin{proposition}\label{not_open}
Let $\varphi: L\to K$ be a surjection of a compact space $L$ onto an infinite space $K$. If $\varphi$ is not injective, then the map $\varphi^*: M_1(L)\to M_1(K)$ is not open.
\end{proposition}
\begin{proof}
We will consider two cases:\smallskip\\
{\bf Case 1.} There exist distinct points $x,y\in K$ such that $|\varphi^{-1}(x)|>1$ and $y$ is an accumulation point of $K$. Pick disjoint neighborhoods $U_x$ of $x$ and $U_y$ of $y$. Take two distinct points $z_0,z_1\in \varphi^{-1}(x)$ and a continuous function $f: L\to [0,1]$ such that $f(z_i)=i$ and $f^{-1}((0,1])\subseteq \varphi^{-1}(U_x)$. Consider the open set $V=\{\mu\in M_1(L): \mu(f)>1/4\}$. We will show that its image $\varphi^*(V)$ is not open in  $M_1(K)$. Clearly, we have
\[\mu = (1/2)\delta_{z_1} - (1/2)\delta_{z_0} \in V \mbox{ and  } \varphi^*(\mu)= 0.\]
 Let $W$ be any open neighborhood of $0$ in $M_1(K)$. Since $y$ is an accumulation point, we can find $y'\in U_y\setminus\{y\}$ such that $\nu = (1/2)\delta_{y} - (1/2)\delta_{y'} \in W$. One can easily check that $\nu\notin \varphi^*(V)$.\smallskip\\
{\bf Case 2.} It is clear that if the Case 1 does not hold, then there exists an accumulation point $x$ in $K$ such that $\varphi^{-1}(x)$ is the only nontrivial fiber of $\varphi$.  Take two distinct points $z_0,z_1\in \varphi^{-1}(x)$ and disjoint neighborhoods $U_0,U_1$ of $z_0,z_1$, respectively, in $L$. Then we have $\varphi(U_0)\cap \varphi(U_1)=\{x\}$, hence there is an $i\in\{0,1\}$ such that $\varphi(U_i)$ is not a neighborhood of $x$. Find a a continuous function $f: L\to [0,1]$ such that $f(z_i)=1$  and $f^{-1}((0,1])\subseteq U_i$. We define the open set $V=\{\mu\in M_1(L): \mu(f)>1/4\}$ as in Case 1. Again, we have
\[\mu = (1/2)\delta_{z_i} - (1/2)\delta_{z_{1-i}} \in V \mbox{ and  } \varphi^*(\mu)= 0.\]
 Take any  open neighborhood $W$ of $0$ in $M_1(K)$. Since $x$ is an accumulation point, we can find distinct points $y,y'\in K\setminus \varphi(U_i)$ such that $\nu = (1/2)\delta_{y} - (1/2)\delta_{y'} \in W$. One can easily verify that $\nu\notin \varphi^*(V)$, hence $\varphi^*(V)$ is not open in  $M_1(K)$.
\end{proof}

\begin{proposition}\label{not_AR_om1}
Let $K$ be a compact space of weight $\omega_1$. Then the space $M_1(K)$ is not a Dugundji space.
\end{proposition}

\begin{proof} Assume towards a contradiction that $M_1(K)$ is a Dugundji space. Then by a result of Haydon, cf. \cite{Ha74}, \cite{Sh76}, $M_1(K)$ is an inverse limit of a continuous inverse sequence $\langle L_\alpha,p^\beta_\alpha,\omega_1\rangle$, where all spaces $L_\alpha$ are metrizable and all bonding maps $p^\beta_\alpha$ are open. Let  $\langle K_\alpha,q^\beta_\alpha,\omega_1\rangle$ be any continuous inverse sequence with all spaces $K_\alpha$ infinite metrizable, all bonding maps $q^\beta_\alpha$ non-injective, and the limit homeomorphic to $K$. Then one can easily verify that the inverse system $\langle M_1(K_\alpha),(q^\beta_\alpha)^*,\omega_1\rangle$ is continuous and its limit is homeomorphic to $M_1(K)$. Then, by Shchepin's spectral theorem \cite{Sh76} the sequences $\langle L_\alpha,p^\beta_\alpha,\omega_1\rangle$ and $\langle M_1(K_\alpha),(q^\beta_\alpha)^*,\omega_1\rangle$ would contain isomorphic subsequences, which is impossible, since the maps $p^\beta_\alpha$ are open and the maps $(q^\beta_\alpha)^*$ are  not open by Proposition \ref{not_open}.
\end{proof}

We need to recall some notions and results from \cite{Ku04} (essentially, these ideas go back to \cite{Sh76}). Let $X=\Pi_{\alpha<\kappa} X_\alpha$ be the product of metrizable compact spaces $X_\alpha$, and let $r:X\to Y$ be a retraction. A subset $S\sub \kappa$ is \emph{$r$-admissible} if $x|S=x'|S$ implies $r(x)|S=r(x')|S$ for all $x,x'\in X$. Obviously the union of any family of $r$-admissible subsets is $r$-admissible.
For $y\in Y$ let $p_S:Y\to \Pi_{\alpha\in S} X_\alpha$ be defined by $p_S(y)=y|S$ and let $Y_S=p_S(Y)$. Each countable subset of $\kappa$ is contained in a countable $r$-admissible subset, and if $S\sub \kappa$ is $r$-admissible then the map $p_S:Y\to Y_S$ is right-invertible, i.e., there exists a continuous map $j: Y_s\to Y$ such that $p_S\circ j = \mathrm{id}_{Y_S}$, hence $Y_S$ is homeomorphic to a retract $j(Y_S)$ of $Y$, cf. \cite{Ku04} and \cite{Sh76}.

\begin{proof}[Proof of Theorem \ref{not_AR}]
Suppose that $K$ is a nonmetrizable compact space such that the space $M_1(K)$ is a Dugundji space, hence an absolute retract. We will show that there exists a continuous image $L$ of $K$ of  weight $\omega_1$ such that $M_1(L)$ is homeomorphic to a retract of $M_1(K)$. This will give a contradiction with Proposition \ref{not_AR_om1}. Let $\mathcal{F}=\{f_t:K\to K_t: t\in T\}$ be the family of all continuous surjections of $K$ onto a subspace of $[0,1]^\omega$. Then the diagonal map
\[ \varphi=\bigtriangleup_{t\in T} f_t^*: M_1(K)\to \Pi_{t\in T} M_1(K_t),\]
is an embedding of $M_1(K)$ into the product of metrizable compacta. Let $Y=\varphi(M_1(K))$ and let $r: \Pi_{t\in T} M_1(K_t)\to Y$ be a retraction. Fix a subset $\{t_\alpha: \alpha<\omega_1\}\sub T$ such that the image of the diagonal map
\[ \bigtriangleup_{ \alpha<\omega_1} f_{t_\alpha}: K\to \Pi_{ \alpha<\omega_1} K_{t_\alpha},\]
is of weight $\omega_1$.
For any countable subset $S\sub T$ fix a countable $r$-admissible subset $\eta(S)\sub T$ containing $S$. By induction we will define the family of $r$-admissible countable sets $S_\alpha\sub T$ for $\alpha<\omega_1$. We start with $S_0=\eta (\{t_0\})$. Suppose that we have defined the sets $S_\beta$ for $\beta<\alpha$. Put $P_\alpha= \bigcup\{S_\beta: \beta<\alpha\}\cup\{t_\alpha\}$ and let $s_\alpha\in T$ be such that $f_{s_\alpha}= \bigtriangleup_{t\in P_\alpha} f_{t}: K\to \Pi_{t\in P_\alpha} K_{t}$. We define $S_\alpha = \eta(P_\alpha\cup\{s_\alpha\})$. Finally we put $S = \bigcup\{S_\alpha: \alpha<\omega_1\}$. The set $S$ is $r$-admissible, hence the map $p_S:Y\to Y_S$ is right-invertible, so $Y_s$ is homeomorphic to a retract of $Y$. Let $L$ be the image of $K$ under the diagonal map $\bigtriangleup_{t\in S} f_{t}: K\to \Pi_{t\in S} K_{t}$. The use of indexes $t_\alpha$ in our construction guaranties that $L$ has weight $\omega_1$. A routine verification shows that $Y_s$ is homeomorphic to $M_1(L)$.
\end{proof}

\section{Countable discrete extensions of dyadic compacta}\label{dyadic}

If $\fA$ is a subalgebra of the algebra of all subsets of $\omega$ containing all finite sets then its Stone space $\ult(\fA)$ can be seen as a compactification of $\omega$
because one can identify every $n\in\omega$ with the corresponding principal ultrafilter.
Note that  $\ult(\fA/\fin)$ is homeomorphic to the remainder of such a compactification.
 In this setting, the equivalence of conditions (i) and (iii) from Lemma \ref{p:2} can be stated as follows (see Lemma 3.1 in \cite{DP16}).

\begin{lemma}\label{d:1}
Let $\fA$ be an algebra such that $\fin\sub\fA\sub P(\omega)$.
Then the compactification $\ult(\fA)$ of $\omega$ is tame if and only if there exists a bounded
sequence $(\nu_n)_n$ in $\ba(\fA)$  such that

\begin{enumerate}[(i)]
\item $\nu_n|\fin \equiv 0$ for every $n$, and
 \item $\nu_n-\delta_n\to 0$ on $\fA$, that is $(\nu_n - \delta_n)(A) \to 0$ for every $A\in\fA$.
 \end{enumerate}
\end{lemma}

A Boolean algebra $\fB$ is called {\em dyadic} if it can be embedded into a free algebra $\clop(2^\kappa)$ for some cardinal number
$\kappa$, that is if $\ult(\fA)$ is a dyadic compactum, i.e., a continuous image of some Cantor cube $2^\kappa$ (\cite{En89}).
Recall that for $L=2^\kappa$ and $L'\in\cde(L)$ there is a retraction from $L'$ onto $L$ so, in particular, there is an extension operator
$C(L)\to C(L')$. We give below examples showing that this is no longer true if we replace here $2^\kappa$ by its continuous image. Let us recall that a family $\mathcal{C}$ of subsets of a set $X$ is {\em independent} if for any disjoint finite $\mathcal{E},\mathcal{F}\subseteq \mathcal{C}$ it holds that $\bigcap_{E\in\mathcal{E}}E\cap\bigcap_{F\in\mathcal{F}}(X\setminus F)\ne\emptyset$.
It is well-known that there is an independent family $\mathcal{E}$ of subsets of $\omega$ such that $|\mathcal{E}|=\con$, see e.g.\ \cite[Lemma 7.7]{Je03}.

\begin{lemma}\label{d:2}
Let $\fB$ be a Boolean algebra generated by a family $\GG$ of size $\kappa$ such that $\GG=\bigcup_n \GG_n$, where
every $\GG_n$ is an independent family and for every $k\neq n$, if $a\in\GG_k$ and $b\in\GG_n$ then $a\cap b=0$.

Then $\fB$ embeds into $\clop(2^\kappa)$.
\end{lemma}

\begin{proof}
Take a pairwise disjoint sequence $D(n)$ in $\clop(2^\kappa)$ and for every $n$ choose independent family
$\{D_\xi(n):\xi<\kappa\}$, where every $D_\xi(n)$ is a clopen subset of $D(n)$.

Write $\GG_n=\{g_\xi(n):\xi<\kappa_n\}$, where $\kappa_n\le\kappa$. Define $\vf$ setting  $\vf(g_\xi(n))=D_\xi(n)$ for
every $n$ and $\xi<\kappa_n$. Then $\vf$ extends to a Boolean embedding $\fA\to \clop(2^\kappa)$ in a obvious way.
\end{proof}

\begin{example}\label{d:3}
There is a dyadic compactum $L$ of weight $\omega_1$ and $L'\in \cde(L)$ such that there is no extension operator
$E:C(L)\to C(L')$ with $\|E\|<2$.
\end{example}

\begin{proof}
Divide $\omega$ into three infinite sets $P, Q_1,Q_2$. Recall that, for $A,B\subseteq\omega$, $A\sub^* B$ means that the set $A\sm B$ is finite;
in such a case we say that $B$ almost contains $A$.

On $P$ we consider a Hausdorff gap, see e.g.\ \cite[Theorem 29.7]{Je03}: take $A_\alpha, B_\alpha\sub P$, $\alpha<\omega_1$ such that

\begin{enumerate}[(a)]
\item $A_\alpha\sub^* A_\beta, B_\alpha\sub^* B_\beta$ for $\alpha<\beta<\omega_1$;
\item $A_\alpha\cap B_\beta$ is finite for every $\alpha,\beta<\omega_1$;
\item there is no $X\sub P$ satisfying $A_\alpha\sub^*X\sub^*P\sm B_\beta$ for $\alpha,\beta<\omega_1$.
\end{enumerate}

For $i=1,2$, we choose a family $\{C_\alpha(i):\alpha<\omega_1\}$ of independent subsets of $Q_i$ and
 define a subalgebra $\fA$  of $P(\omega)$ generated by $\fin$ and all the sets
\[G_\alpha(1)= C_\alpha(1)\cup A_\alpha,   G_\alpha(2)=C_\alpha(2)\cup B_\alpha, \alpha<\omega_1.\]
By Lemma \ref{d:2} the algebra $\fA/\fin$ is dyadic.

Now $L'=\ult(\fA)$ is a countable discrete extension of  $L=\ult(\fA/\fin)$.
%\medskip
 Suppose that there is an extension operator $E:C(L)\to C(L')$ such that $r=\|E\|<2$.
 Take a sequence $(\nu_n)_n$ in $\ba(\ult(\fA))$ as in Lemma \ref{p:2}(iii). In the sequel, we treat every $\nu_n$ as a member of $M(\fA)$, cf.\
 the remarks at the end of Section 2.

Then $\|\nu_n\|\le r<2$ for every $n$. Take $\delta>0$ such that $r<2-2\delta$.
For every $\alpha<\omega$ put
\[ \widehat{A}_\alpha=\{n\in A_\alpha: \nu_n(G_\alpha(1))>1-\delta\}.\]
Then $A_\alpha\sub^*\widehat{A}_\alpha$ since $\lim_{n\in A_\alpha} \nu_n(G_\alpha(1))=1$. Hence the set $X=\bigcup_{\alpha<\omega_1} \widehat{A}_\alpha$
almost contains every $A_\alpha$. On the other hand, for every $\beta<\omega_1$, $B_\beta\cap X$ must be finite: otherwise,
there is $n\in B_\beta\cap X$ such that $\nu_n(G_\beta(2))>1-\delta$. Since $n\in X$, $n\in \widehat{A}_\alpha$ for some $\alpha$ so
$\nu_n(G_\alpha(1))>1-\delta$. But $G_\alpha(1)\cap G_\beta(2)$ is finite so  $\nu_n(G_\alpha(1)\cap G_\beta(2))=0$.
It follows that
\[ \|\nu_n\|\ge \nu_n(G_\alpha(1))+\nu_n(G_\beta(2))>2-2\delta>r,\]
 contrary to our assumption.

 In this way we have checked that $X$ separates the gap, which is impossible.
\end{proof}

It is a well-known fact from the theory of absolute retracts that a metrizable compactum $M$ is an absolute retract, provided it is a union of two compact absolute retracts $M_1, M_2$ whose intersection $M_1\cap M_2$ is also an absolute retract. It is also known that this is not the case without the metrizability assumption. Our Example \ref{d:3} can be applied to demonstrate this.

\begin{corollary}\label{AR_union} Let $K = 2\times [0,1]^{\omega_1}$, and $x$ be a fixed point of $[0,1]^{\omega_1}$. The quotient space $M$ obtained from $K$ by identification of the points $(0,x)$ and $(1,x)$ is the union of two copies of $[0,1]^{\omega_1}$ intersecting at a single point, yet it is not an absolute retract.
\end{corollary}
\begin{proof}
We adopt the notation from the proof of Example \ref{d:3}.

Observe that, since the cube $[0,1]^{\omega_1}$ is homogeneous, the space $M$ does not depend on the choice of a point $x$. We can assume that $x\in\{0,1\}^{\omega_1}$. Let $S$ be the subspace of $M$ which is a quotient image of $2\times \{0,1\}^{\omega_1}\subseteq K$. Using the fact that $\{0,1\}^{\omega_1}$ is a Dugundji space, we can easily obtain an extension operator $E':C(S)\to C(M)$ of norm 1.

One can easily verify that $S$ is homeomorphic to the space $L$ from Example \ref{d:3}. Indeed, for $i=1,2$, let  $\fA_i$ be the subalgebra of $P(\omega)$ generated by $\fin$ and the family of sets $G_\alpha(i)$,  $\alpha<\omega_1$. These families are independent, hence $L_i=\ult(\fA_i/\fin)$ are homeomorphic to $\{0,1\}^{\omega_1}$. Since all intersections $G_\alpha(1)\cap G_\beta(2)$ are finite, we conclude that  $L$ is homeomorphic to $S$. Therefore, we can take $S'\in \cde(S)$ such that there is no extension operator
$E:C(S)\to C(S')$ with $\|E\|<2$. We can assume that $S'\setminus S$ is disjoint from $M$. Let $M'= M\cup S'$. If there was a retraction $r:M'\to M$, then the assignment
$f\mapsto E'(f)\circ (r|S')$ would define an extension operator from  $C(S)$ to $C(S')$ of norm 1, a contradiction.
\end{proof}

\begin{proposition}
Let $K$  be a compact space, such that, for some point $p\in K$, $K=\bigcup_{i=1}^n K_i$, where $K_i$ is a Dugundji space, and $K_i\cap K_j = \{p\}$, for all $i,j\le n, i\ne j$. Then, for any compact space $L$ containing $K$, there exists an extension operator $E:C(K)\to C(L)$ with $\|E\|\le 2n-1$.
\end{proposition}
\begin{proof}
For any $i\le n$, let $L_i$ be the quotient space obtained from $L$ by identifying all points from $\bigcup_{j\ne i} K_j$ with the point $p$, and let $q_i: L\to L_i$ be the quotient map. Clearly, $q_i$ maps $K_i$ homeomorphically onto $q_i(K_i)$. Let $r_i: q_i(K_i)\to K_i$ be the inverse homeomorphism.
By our assumption on $K_i$, we can find an extension operator $E_i: C(q_i(K_i)) \to C(L_i)$ of norm $1$. Now, we can define the extension operator $E:C(K)\to C(L)$ by
\[ E(f)(x)= \sum_{i=1}^{n} E_i(f|K_i\circ r_i)(q_i(x)) - (n-1)f(p),\]
for $f\in C(K)$ and $x\in L$. It is clear that, for each $f\in C(K)$, $E(f)$ is continuous on $L$. If $x\in K$, then $x\in K_i$, for some $i$, hence $q_j(x)=p$ and $ E_j(f|K_j\circ r_j)(q_j(x))=f(p)$ for $j\ne i$. Therefore $E(f)(x)= f(x)$. Obviously, we have $\|E\|\le 2n-1$, so $E$ is as desired.
\end{proof}

From the above Proposition and the proof of Corollary \ref{AR_union} we immediately obtain the following

\begin{corollary}
For the spaces $L$ and $L'$ from Example \ref{d:3} there exists an extension operator $E:C(L)\to C(L')$ of  norm 3.
\end{corollary}

Our next example uses the concept of multigaps introduced by Avil\'es and \stevo\  and partially builds on Theorem 29 from \cite{AT11}.
%For the next example we need to recall some stuff from \cite{AT11}.
In what follows, we consider ideals $\II$ on $\omega$ containing all finite sets.
Two ideals $\II_1,\II_2$ are {\em orthogonal} if $A_1\cap A_2$ is finite for any $A_i\in\II_i$. Given $k$,
by a {\em $k$-gap} we mean a family $\{ \II_1,\ldots, \II_k\}$ of mutually orthogonal
ideals such that:

\begin{quote}
\em
Whenever $X_1,\ldots, X_k\sub\omega$ and $A\sub^*X_i$, for every $i\le k$ and $A\in\II_i$,
then
$\bigcap_{i\le k} X_i\neq\emptyset$.
\end{quote}

Note that a Hausdorff gap (see Example \ref{d:3}) is, in particular,  a $2$-gap defined by ideals generated by $\omega_1$ sets. Avil\'es and Todorcevic \cite{AT11} proved that for every $k$ there are
$k$-gaps of $\con$-generated ideals; on the other hand, under $\MA(\omega_1)$   there are no $3$-gaps defined by $\omega_1$-generated ideals.

\begin{example}\label{d:4}
There is a dyadic compactum $L$ of weight $\con$ and $L'\in\cde(L)$ such that there is no extension operator
$E:C(L)\to C(L')$.
\end{example}

\begin{proof}
Take a partition $\omega=\bigcup_{k\ge 2} N_k$ into infinite sets. For every $k\ge 2$ divide $N_k$ into infinite sets $P_k$, $Q_{k,j}$, $j\le k$.
Let $\II(k,j)$, $j\le k$ be a family of mutually orthogonal ideals of subsets of $P_k$ that constitutes a $k$-gap.

Fix $k$ and $j\le k$. Choose an independent family $\mathcal{C}(k,j)=\{C_\xi(k,j):\xi<\con\}$ of subsets of $Q_{k,j}$ and  fix some enumeration
$\{I_\xi(k,j):\xi<\con\}$ of $\II_{k,j}$.

We define $\fA$ to be an algebra
of subsets of $\omega$ generated by finite sets and
\[ G_\xi(k,j)=I_\xi(k,j)\cup C_\xi(k,j), \xi<\con, k\ge 2, j\le k.\]
By Lemma \ref{d:2} $\fA/\fin$ is a dyadic algebra (can be embedded into $\clop(2^\con)$), so $L=\ult(\fA/fin)$ is a dyadic compactum
of weight $\le\con$.
We let $L'=L\cup\omega$ which is identified with $\ult(\fA)$.

Suppose that there is an extension operator $E:C(L)\to C(L')$. Take a sequence $(\nu_n)_n$ in $\ba(\fA)$ as in Lemma \ref{d:1}.
Then $\|\nu_n\|\le \|E\|$ for every $n$. Take $k>2\cdot \|E\|$.

Note that for every $j\le k$ and $A\in \II(k,j)$ there is $C_A\in\mathcal{C}(k,j)$ such that $A\cup C_A\in\fA$; moreover,
if $A\in \II(k,j)$ and $A'\in \II(k,j')$ with $j\neq j'$ then $C_A\cap C_{A'}=\emptyset$.
For $A\in \II(k,j)$ put
\[\widehat{A}=\{n\in A: \nu_n(A\cup C_A)>1/2\}.\]

Then $A\sub^*\widehat{A}$ since $\lim_{n\in A} \nu_n(A\cup C_A)=1$.
Hence the set
\[ X_j=\bigcup_{A\in\II(k,j)} \widehat{A},\]
almost contains every $A\in\II(k,j)$.
Since the family $\{\II(k,j):j\le k\}$ constitutes a $k$-gap, there is  $n\in \bigcap_{j\le k}X_j$.
Then there are $A_j\in\II(k,j)$, $j\le k$,  such that $n\in \widehat{A}_j$ so $\nu_n(A_j\cup C_{A_j})>1/2$ and
$A_j\cup C_{A_j}$ are almost pairwise disjoint for different $j$'s. Since $\nu_n$ vanishes on $\fin$, this gives $\|\nu_n\|> k/2> \|E\|$, a contradiction.
\end{proof}

\begin{problem}\label{d:5}
Can we, in ZFC,
define $L$ as in Example \ref{d:4}, but of weight $\omega_1$?
\end{problem}

Correa and Tausk \cite{CT16} proved  that if a compact space $K$ contains a copy of $2^\con$, then $C(K)$ admits a nontrivial twisted sum with $c_0$. Gerlits and Efimov showed that every dyadic compactum $K$ contains a
copy of the Cantor cube $2^\kappa$, for every regular
cardinal number $\kappa\le w(X)$, see \cite[3.12.12]{En89}. From these results easily follows

\begin{theorem}\label{d:6}
Assuming $\CH$, for each nonmetrizable dyadic space $K$, $c_0$ and $C(K)$ have a nontrivial twisted sum.
\end{theorem}

\section{Linearly ordered compact spaces}\label{lo}

The following theorem is a consequence  of  several known results.

\begin{theorem}\label{nonseparable_LOTS}
Assuming $\CH$,  if $L$ is a nonseparable linearly ordered compact space, then there is a nontrivial twisted sum of $c_0$ and $C(L)$.
\end{theorem}

\begin{proof}
Recall that every measure on a linearly ordered compactum has a separable support, see \cite{Sa80} or \cite{Me96}.
Hence the space $L$ does not support a strictly positive measure.

If $L$ has $ccc$ then it is first-countable (\cite[3.12.4]{En89}), and it follows that $|L|\le\con$ (\cite[3.12.11(d)]{En89}) so
$L$ is of  weight $\omega_1=\con$. Therefore we obtain the desired conclusion by Theorem  \ref{Kubis}.

If $L$ does not satisfy $ccc$ then the assertion of the theorem  follows from a result of Correa and Tausk stated in \cite{CT16}.
Namely, it is well-known that in such a case  $C(L)$ contains an isometric copy of $c_0(\omega_1)$, and by \cite[Corollary 2.7]{CT15}
this copy is complemented in $C(L)$. It remains to recall that there exists a nontrivial twisted sum of $c_0$ and $c_0(\omega_1)$, cf.\ \cite{Ca16}.
\end{proof}

The next result is a far-reaching generalization of Theorem 7.1 from \cite{DP16}.

\begin{theorem}\label{separable_LOTS}
Let $L$ be a separable linearly ordered compact space of weight $\kappa$ such that $2^\kappa>\con$. Then there is a non-tame compactification $\gamma\omega$ with remainder homeomorphic to $L$. Hence there is a nontrivial twisted sum of $c_0$ and $C(L)$.
\end{theorem}

\begin{corollary}\label{separable_LOTS_c}
If $L$ is a separable linearly ordered compact space of weight $\con$, then there is a nontrivial twisted sum of $c_0$ and $C(L)$.
\end{corollary}

\begin{corollary}\label{LOTS_CH}
Under  $\CH$,  if $K$ is a nonmetrizable linearly ordered compact space, then there is a nontrivial twisted sum of $c_0$ and $C(K)$.
\end{corollary}

Note that Corollary \ref{LOTS_CH} follows directly   from Corollary \ref{separable_LOTS_c}  and Theorem \ref{nonseparable_LOTS}.
The rest of this section is devoted to proving Theorem \ref{separable_LOTS}.

We shall use the following well-known description of the class of separable linearly ordered compact spaces.
Let $A$ be an arbitrary subset of a closed subset $K$ of the unit
interval $I=[0,1]$. Put
\[ K_{A}=(K\times \{0\})\cup (A\times\{1\}),\]
and equip this set with the order topology given by the
lexicographical order (i.e., $(s,i)\prec(t,j)$ if either $s<t$, or
$s=t$ and $i<j$).

For $K=I$ and $A=(0,1)$ the space $\mathbb{K}=K_A$ is a well known double arrow space (some authors use this name for the space $I_I$).

It is known that the class of all spaces $K_A$
coincides with the class of separable linearly ordered compact
spaces. Namely, the following is a reformulation of the characterization due to Ostaszewski \cite{Os74}:

\begin{theorem}[Ostaszewski]\label{Ostaszewski} The space $L$ is a separable compact linearly
ordered space if and only if $L$ is homeomorphic to $K_A$ for some
closed set $K\subseteq I$ and a subset $A\subseteq K$.
\end{theorem}

The next lemma seems to belong to the mathematical folklore, we include a short justification for the readers convenience.

\begin{lemma}\label{LOTS_subspace}
Let $L$ be a separable linearly ordered compact space of uncountable weight $\kappa$. Then $L$ contains a topological copy of the space $I_B$, where $B$ is a dense subset of $(0,1)$ of the cardinality  $\kappa$.
\end{lemma}
\begin{proof}
By Theorem \ref{Ostaszewski} we can assume that $L=K_A$ for some closed $K\subseteq I$ and some subset $A$ of $K$. From our assumption on the weight of $K$ it easily follows that $|A|=\kappa$. Take a dense-in-itself subset $C$ of $A$ of cardinality $\kappa$. Let $M$ be the closure of $C$ in $K$ and let $a=\inf M, b=\sup M$. Put $D= M\cap A\cap(a,b)$. Obviously, $D$ is a dense subset of $M$ of the cardinality $\kappa$ and the space $M_D$ can be identified with a subspace of $K_A$. Let $\{(a_n,b_n): n<m\}$ be an enumeration of the family of all components of $[a,b]\setminus M$ for some $m\le\omega$. Put
\[ P=M_D\setminus\left(\{(a_n,1): a_n\in D\}\cup\{(b_n,0): b_n\in D\}\right).\]
Then $P$ is a closed, dense-in-itself subspace of $M_D$. Let $\sim$ be the equivalence relation on $M$ defined by $a_n\sim b_n$ for $n<m$, and let $q:M\to M_{/\sim}$ be the quotient map. The space $S= M_{/\sim}$ is compact, linearly ordered, connected, and metrizable, hence there is a  homeomorphism $h: M_{/\sim}\to I$ with $h(q(a))=0$. One can easily verify that $P$ can be identified with $I_B$ where $B=h(q(D\cup\{a_n: n<m\}))$.
\end{proof}

\begin{theorem}\label{I_B}
Let $B$ be a dense subset of $(0,1)$ of the cardinality  $\kappa$ such that $2^\kappa>\con$. Then there is a non-tame compactification $\gamma\omega$ which remainder is homeomorphic to $I_B$.
\end{theorem}
\begin{proof}
Let $Q$ be a countable dense subset of $(0,1)$. For each $x\in B$ put $P_x=\{q\in Q: q\le x\}$ and pick a strictly increasing sequence $(q_x^n)_{n\in\omega}$ in $Q$ such that $\lim_{n}q_x^n = x$. Let $S_x=\{q_x^n: n\in\omega\}$. For any $f: B\to 2$ define
\begin{equation*}
R_x^f =\begin{cases} P_x& \text{if $f(x)=0$},\\ P_x\setminus S_x & \text{if $f(x)=1$}.
\end{cases}
\end{equation*}
Let $\mathcal{A}^f$ be a subalgebra of $P(Q)$ generated by $\{R_x^f: x\in B\}\cup\fin$, where $\fin$ denotes the family of all finite subsets of $Q$. We shall check that, for any $f$, the Stone space $\ult(\mathcal{A}^f)$ is a compactification of a countable discrete space with remainder homeomorphic to $I_B$.

For any $q\in Q$ let $u_q^f$ denote the ultrafilter in $\ult(\mathcal{A}^f)$ containing $\{q\}$. Let $r^f: \mathcal{A}^f \to \mathcal{A}^f/\fin$ be the quotient map. It is well-known that $\ult(\mathcal{A}^f)$ is a compactification of its countable discrete subspace $\{u_q^f: q\in Q\}$ and its remainder can be identified with the space $\ult(\mathcal{A}^f/\fin)$. Observe that $\mathcal{A}^f/\fin$ is generated by the family  $\{r^f(R_x^f): x\in B\}$. For $x,y\in B$, $x<y$ we have $P_x\subseteq P_y$, the difference $P_y\setminus P_x$ is infinite, and the intersection $P_x\cap S_y$ is finite. Therefore
\[
R^f_x\subsetneq^* R^f_y \Leftrightarrow x<y
\mbox{ for } x,y\in B \mbox{  and } f\in 2^B.\]

Let $U$ be an ultrafilter in ${\cA}^f/\fin$. The set $T_U=\{x\in B: r(R_x^f)\in U\}$ is a final segment in $(B,<)$, hence, either
\begin{eqnarray*}
(\exists z\in I\setminus B) \quad T_U &=& (z,1)\cap B \quad\text{or}\\
(\exists y\in B) \quad T_U &=& [y,1)\cap B \quad\text{or}\\
(\exists y\in B) \quad T_U &=& (y,1)\cap B\,.
\end{eqnarray*}
The ultrafilter $U$ is uniquely determined by the set $T_U$. For $z\in I\setminus B$ let $U_z^f$ be the ultrafilter in $\mathcal{A}^f/\fin$ such that $T_{U_z^f} =(z,1)\cap B$.
For $y\in B$ let $U_{y,0}^f$, $U_{y,1}^f$ be the ultrafilters such that $T_{U_{y,0}^f}=[z,1)\cap B$, $T_{U_{y,1}^f}=(z,1)\cap B$.

A routine verification shows that the map $\varphi^f: \ult(\mathcal{A}^f/\fin) \to I_B$ given by
\begin{equation*}
\varphi^f(U) = \begin{cases} (z,0)& \text{if\quad $U = U_z^f,\quad z\in I\setminus B$},\\
(y,i)& \text{if\quad $U = U_{y,i}^f,\quad y\in B, i=0,1$}
\end{cases}
\end{equation*}
is a homeomorphism.
The map $\psi^f: \ult(\mathcal{A}^f/\fin) \to \ult(\mathcal{A}^f)$, given by $\psi^f(U) = (r^f)^{-1}(U)$, for $U \in  \ult(\mathcal{A}^f/\fin)$ is a homeomorphic embedding. Let
\[ u_z^f = \psi^f(U_z^f),  u_{y,i}^f = \psi^f(U_{y,i}) \mbox{  for  } z\in I\setminus B, y\in B, i=0,1.\]
 Observe that if $f(x)=0$, then $S_x\subseteq R_x^f$, otherwise  $S_x\subseteq Q\setminus R_x^f$, hence $S_x$ is contained in the element of the ultrafilter $u_{y,f(x)}^f$. It follows that the sequence $(u^f_{q_x^n})_{n\in\omega}$ converges to $u_{y,f(x)}^f$ in  $\ult(\mathcal{A}^f)$.
The space $M(I_B)$ has the cardinality $\con$, which follows for instance from the fact that every probability measure
on $I_B$ has a uniformly distributed sequence, see Mercourakis \cite{Me96}.
Hence the family $\mathcal{E}$ of all maps $e:Q\to M(I_B)$ has the same cardinality.

Suppose that for all $f\in 2^B$ there is an extension operator
\[ T^f: C(\psi^f(\ult(\mathcal{A}^f/\fin))) \to C(\ult(\mathcal{A}^f)).\]
Consequently, by Lemma \ref{p:2} there exists a continuous map $g^f: \ult(\mathcal{A}^f)\to M(\psi^f(\ult(\mathcal{A}^f/\fin)))$ such that, for any $U\in \ult(\mathcal{A}^f/\fin)$, $g^f(\psi^f(U)) = \delta_{\psi^f(U)}$. By continuity of $g^f$ we have
\[ \lim_n g^f(u^f_{q_x^n}) = \delta{u_{y,f(x)}^f},\]
for all $x\in B$. Let
\[ \phi^f: M(\psi^f(\ult(\mathcal{A}^f/\fin))) \to M(I_B),\]
be the isometry induced by the embedding $\psi^f$ and the homeomorphism $\varphi^f$. Let $e^f: Q\to M(I_B)$ be defined by $e^f(q) = \phi^f(g^f(u^f_q))$. Then, for any $x\in B$, the sequence $(e^f(q^n_x))_{n\in\omega}$ converges to $\delta_{(x,f(x))}$. It follows that the assignment $f\mapsto e^f$ is an injection of $2^B$ into $\mathcal{E}$, a contradiction.
\end{proof}

Theorem \ref{I_B} immediately implies the following.

\begin{corollary}
Let $\mathbb{K}$ be the double arrow space. There is a non-tame compactification $\gamma\omega$ which remainder is homeomorphic to $\mathbb{K}$.
Hence there is a nontrivial twisted sum of $c_0$ and $C(\mathbb{K})$.
\end{corollary}

 Let us remark that if $K$ is a closed subset of a linearly ordered compact space $L$, then there is a regular extension operator $E: C(K)\to C(L)$, cf.\ \cite{HL74}. Using this fact one can easily deduce Theorem \ref{separable_LOTS} from Theorem \ref{I_B},  Lemma \ref{LOTS_subspace}, Remark \ref{p:11} and Proposition \ref{p:12}.

\section{On scattered compact spaces}\label{sc}

We  start this section by presenting one more construction of non-tame compactifications of $\omega$,
based on an idea similar to that  used in the proof of Theorem \ref{I_B}.
For a compact space $K$ we denote by $\mathrm{Auth}(K)$ the group of autohomeomorphisms of $K$.

\begin{theorem}\label{auth}
Let $L$ be a compact space such that
\begin{enumerate}[(i)]
\item $|M(L)|=\con$,
\item $L$ contains a continuous image $K$ of $\omega^*=\beta\omega\sm\omega$ such that $|\mathrm{Auth}(K)|>\con$.
\end{enumerate}
Then there exists a countable discrete extension $L'$ of $L$ such that there is no extension operator $E:C(L)\to C(L')$;
in particular, there is a nontrivial twisted sum of $C(L)$ and $c_0$.

Moreover, if $L=K$ then we may additionally require that $L'$ is a non-tame compactification of $\omega$ which remainder is homeomorphic to $L$.
\end{theorem}

\begin{proof}
The assumption that $K$ is a  continuous image of $\omega^*$ is equivalent to the existence of a compactification $\gamma\omega$ of $\omega$ with the remainder  $\gamma\omega\setminus\omega$ homeomorphic to $K$. We can assume that $\gamma\omega$ and $L$ are disjoint. Let $S=L\cup \gamma\omega$ be the disjoint
union of $L$ and $\gamma\omega$.  Consider the family  $\mathcal{H}$ of all homeomorphisms $\varphi: \gamma\omega\setminus\omega\to K$. By our assumption $|\mathcal{H}|>\con$. For any $\varphi\in \mathcal{H}$, let $\sim_\varphi$ be the equivalence relation on $S$ given by $x \sim_\varphi \varphi(x)$ for all $x\in \gamma\omega\setminus\omega$ and let $q_\varphi: S\to S/_{\sim_\varphi}$ be a corresponding quotient map. Clearly $q_\varphi(L)$ is homeomorphic to $L$ and $S/_{\sim_\varphi}$ is a countable discrete extension of $q_\varphi(L)$.

Suppose that for all $\varphi\in \mathcal{H}$ there is an extension operator
\[T_\varphi: C(q_\varphi(L))\to C(S/_{\sim_\varphi}).\]
 Consequently, by Lemma \ref{p:2} for every such $\vf$ there exists a continuous map
 \[ h_\varphi: S/_{\sim_\varphi}\to M(q_\varphi(L)),\]
satisfying  $h_\varphi(y) = \delta_{y}$ for every  $y\in q_\varphi(L)$. Let $g_\varphi: M(q_\varphi(L))\to M(L)$ be the isometry induced by the homeomorphism $q_\varphi|L: L\to q_\varphi(L)$. Define $e_\varphi:\omega\to M(L)$ by $e_\varphi(n) = g_\varphi(h_\varphi(q_\varphi(n)))$ for $n\in\omega$. Observe that the family of all maps from $\omega$ to $M(L)$ has the cardinality $\con$, since $|M(L)|=\con$. We will get the desired contradiction by showing that the assignment $\varphi\mapsto e_\varphi$ is one-to-one. Fix distinct $\varphi_0,\varphi_1\in\mathcal{H}$, and take $x\in \gamma\omega\setminus\omega$ such that $\varphi_0(x)\ne \varphi_1(x)$. Pick $f\in C(L)$ with $f(\varphi_i(x))=i$, $i=0,1$. Since $q_{\varphi_i}(x) = q_{\varphi_i}(\varphi_i(x)))$, we have $g_{\varphi_i}(h_{\varphi_i}(q_{\varphi_i}(x)))(f)=i$. By continuity of $h_{\varphi_i}$ we can find a neighborhood $U_i$ of $x$ in $\gamma\omega$ such that
\[ g_{\varphi_0}(h_{\varphi_0}(q_{\varphi_0}(z)))(f)<1/2  \mbox{ for } z\in U_0; \]
\[ g_{\varphi_1}(h_{\varphi_1}(q_{\varphi_1}(z')))(f)>1/2  \mbox{ for } z'\in U_1.\]
Now, for any $n\in\omega$ such that $n\in U_0\cap U_1$, we have $e_{\varphi_0}(n)(f)<1/2<e_{\varphi_1}(n)(f)$.
\end{proof}

We shall now consider the well-known class of compact spaces associated with uncountable almost
disjoint families of subsets of $\omega$ --- separable compact spaces $K$ whose set of accumulation points is the one-point compactification of an uncountable discrete space. Such compact spaces were
considered first by Aleksandrov and Urysohn \cite{AU29} and for that reason we
call them \textsf{AU}-compacta, cf.\ \cite{MP09}. It is worth recalling that the space $C(K)$ for    an  \textsf{AU}-compactum $K$
may have interesting structural properties, see Koszmider \cite{Ko05}.

In Section \ref{tts} we considered an \textsf{AU}-compactum described as the Stone space of the algebra of subsets of $\omega$ generated by
finite sets and a given almost disjoint family.
Below we  use the following description of \textsf{AU}-compacta.

Let $D$ be an infinite countable set and let $\mathcal{A}$ be an uncountable
almost disjoint family of infinite subsets of $D$, i.e.\  the
intersection of any two distinct members of $\mathcal{A}$ is
finite.
Let $A\mapsto p_A$ be a one-to-one correspondence between members
of $\mathcal{A}$ and points in some fixed set disjoint from
 $D$, an let $\infty$ be a point distinct from points in  $D$ and
 any point $p_A$. In the set
\begin{eqnarray*}\label{def_AU}
K_\mathcal{A}= D\cup\{p_A\colon A\in\mathcal{A}\} \cup \{\infty\}
\end{eqnarray*}
we introduce a topology declaring that points of $D$ are isolated,
basic neighborhoods $p_A$ are of the form $\{p_A\}\cup(A\setminus
F)$, where $F\subset A$ is finite, and $\infty$ is the point at
infinity of the locally compact space $D\cup\{p_A\colon
A\in\mathcal{A}\}$.

From Theorem \ref{auth} one can easily derive the following

\begin{corollary}\label{cor_AU}
Let $\mathcal{A}$ be an  almost
disjoint family of subsets of $\omega$ of cardinality $\kappa$, where $2^\kappa > \con$. Then there exists a non-tame compactification $\gamma\omega$ which remainder is homeomorphic to $K_\mathcal{A}$.
Hence there is a nontrivial twisted sum of $c_0$ and $C(K_\mathcal{A})$.
\end{corollary}

\begin{proof} First, recall that every measure from $M(K_\mathcal{A})$ is purely atomic, hence  $|M(K_\mathcal{A})|= \con$. Second, observe that the subspace $K=\{p_A\colon A\in\mathcal{A}\} \cup \{\infty\}$ of $K_\mathcal{A}$ is homeomorphic to a one point compactification of a discrete space of cardinality $\kappa$, therefore $|\mathrm{Auth}(K)|= 2^\kappa>\con$. Next, notice that $K$ is a  continuous image of $\omega^*$, since $K_\mathcal{A}$ is a compactification of $\omega$ with remainder $K$. Finally, we can obtain the desired non-tame  compactification using Theorem \ref{auth} and Proposition \ref{p:12}.
\end{proof}
\medskip

Recall that a space $X$ is \emph{scattered} if no nonempty subset $A\subseteq X$ is dense-in-itself.
For an ordinal $\alpha$, $X^{(\alpha)}$ is the $\alpha$th
Cantor-Bendixson derivative of the space $X$. For a scattered
space $X$, the scattered height
${ht(X)}=\min\{\alpha: X^{(\alpha)} =
\emptyset\}$.

Using Theorem \ref{auth} we can also provide an alternative, more topological proof of the following result of Castillo.

\begin{theorem}[Castillo \cite{Ca16}]\label{Castillo}
Under $\CH$,  if  $K$ is a nonmetrizable scattered compact space of finite height, then there exists a nontrivial twisted sum of $c_0$ and $C(K)$.
\end{theorem}

Our argument is based on the following two auxiliary facts.

\begin{proposition}\label{scattered_retract}
Each nonmetrizble scattered compact space $K$  contains a nonmetrizble retract of cardinality at most $\con$.
\end{proposition}

\begin{proof} Consider the family of all uncountable (equivalently,  nonmetrizable) clopen subspaces of $K$, and pick such a subspace $L$ of minimal height $\alpha$. By compactness of $L$, $\alpha$ is a successor ordinal, i.e., $\alpha=\beta+1$. The set $L^{(\beta)}$ is finite, therefore we can partition $L$ into finitely many clopen sets containing exactly one point from $L^{(\beta)}$. One of these sets must be uncountable, hence, without loss of generality we can assume that $L^{(\beta)}=\{p\}$. For every $x\in L\setminus \{p\}$ fix a clopen neighborhood $U_x$ of $x$ in $L$ such that $p\notin U_x$. Clearly, the height of $U_x$ is less than $\alpha$, so, by our choice of $L$, $U_x$ must be countable, hence metrizable. Since every point of $L\setminus \{p\}$ has a metrizable neighborhood it follows that
\begin{equation}\label{s_r1}
(\forall  A\subset L)\,  (\forall  x\in \overline{A}\setminus \{p\})\, (\exists (x_n))\quad x_n\in A \mbox{ and } x_n\rightarrow x\,.
\end{equation}

For any subset $A\subseteq L\setminus \{p\}$ define
\[\varphi(A) = \overline{\bigcup\{U_x: x\in A\}}\setminus \{p\}\,.\]

Observe that, by $|U_x|\le\omega$ and  (\ref{s_r1}), we have
\begin{equation}\label{s_r2}
|\varphi(A)|\le \con, \mbox{ provided } |A|\le \con\,.
\end{equation}

Fix any subset $A\subseteq L\setminus \{p\}$ of cardinality $\omega_1$. We define inductively, for any $\alpha<\omega_1$, sets  $A_\alpha\subseteq L\setminus \{p\}$. We start with $A_0=A$, and at successor stages we put $A_{\alpha+1}= \varphi(A_\alpha)$. If $\alpha$ is a limit ordinal we define  $A_\alpha = \bigcup\{A_\beta: \beta<\alpha\}$. Finally we take $B = \bigcup\{A_\alpha: \alpha<\omega_1\}$. From (\ref{s_r2}) we conclude that $|B|\le \con$. First, observe that $B$ is open in $L$, since, for any $x\in B$, $x$ belongs to some $A_\alpha$, and then $U_x\subseteq A_{\alpha+1}\subseteq B$. Second, the union $M= B\cup \{p\}$ is closed in $L$. Indeed, if $x\in \overline{B}\setminus \{p\}$, then by (\ref{s_r1}), there is a sequence $(x_n)$ in $B$ converging to $x$. We have $\{x_n: n\in\omega\}\subseteq A_\alpha$, for some $\alpha<\omega_1$, therefore $x\in A_{\alpha+1}\subseteq B$. Now, we can define a retraction $r: L\to M$ by
\[ r(x)= \begin{cases} x& \mbox{ for } x\in M,\\
p& \mbox{ for } x\in L\setminus M\,.
\end{cases}\]
Then $r$ is continuous since it is continuous on closed sets $M$ and $L\setminus B$. It remains to observe that $M$ is also a retract of $K$, since $L$ is  a retract of $K$.
\end{proof}

\begin{lemma}\label{scattered_one_point}
Every nonmetrizable scattered compact space $K$ of finite height contains a copy of a one point compactification of an uncountable discrete space.
\end{lemma}

\begin{proof} Let $n+1$ be the height of $K$. Using the same argument as at beginning of the proof of Proposition \ref{scattered_retract}, without loss of generality, we can assume that $K^{(n)}=\{p\}$ and every $x\in K\setminus \{p\}$ has a countable clopen neighborhood $U_x$ in $K$. Let $k=\max\{i: |K^{(i)}|>\omega\}$. Consider
\[ A= K^{(k)}\setminus(\bigcup\{U_x: x \in  K^{(k+1)}\setminus \{p\}\}\cup\{p\})\,.\]
Observe that by our choice of $k$, the set $A$ is uncountable. One can easily verify that the set $A$ is discrete and $p$ is the unique accumulation point of $A$. Therefore $L=A\cup \{p\}$ is a one point compactification of an uncountable discrete space.
\end{proof}

\begin{proof}[Proof of Theorem \ref{Castillo}]
Let $K$ be a nonmetrizable scattered compact space of finite height, and let $L$ be a a nonmetrizable retract of $K$ of  cardinality at most $\con$, given by Proposition \ref{scattered_retract}. Obviously, in the presence of continuum hypothesis, we have $|L|=\con$. Since $L$ is a retract of $K$ it is enough to justify the existence of a nontrivial twisted sum of $c_0$ and $C(L)$. Take a copy $S$ in $L$ of a one point compactification of an uncountable discrete space, given by
Lemma \ref{scattered_one_point}. Obviously, we have  $|S|=\con$. Since $L$ is scattered, every measure in $M(L)$ is purely atomic, hence  $|M(L)|= \con$. We also have $|\mathrm{Auth}(S)|= 2^{\con}$, so we can apply Theorem \ref{auth} as in the proof of Corollary \ref{cor_AU}.
\end{proof}

Clearly, a compact scattered space supports measure if and only if it is separable. Therefore we have the following easy consequence of Lemma \ref{Kubis}.

\begin{corollary}\label{scattered_nonsep}
If $K$ is a nonseparable scattered compact space of weight $\omega_1$, then there exists a nontrivial twisted sum of $c_0$ and $C(K)$.
\end{corollary}

Theorem \ref{Castillo}, Corollary \ref{cor_AU}, and Corollary  \ref{scattered_nonsep} should be compared with the following direct consequence of Corollary \ref{tts:3}.

\begin{theorem}\label{added}
Assume $\MA(\kappa)$ and  let $K$ be a separable scattered compact space of height $3$ and weight $\kappa$. Then every twisted sum of  $c_0$ and $C(K)$ is trivial.
\end{theorem}

\begin{proof}
It is well-known that each infinite scattered compact space $K$ contains a nontrivial convergent sequence, and hence  $C(K)$ contains a complemented copy of $c_0$.
Consequently, for any $n\in\omega$, the space $C(K)$ is isomorphic with $C(K)\oplus \mathbb{R}^n$.

If $K$  is a separable scattered compact space of height 3, then the quotient space $L$ obtained from $K$ by gluing together all points in $K^{(2)}$ is an \textsf{AU}-compactum. Let $|K^{(2)}|=n$. A standard factorization argument shows that $C(K)$ is isomorphic to $C(L)\oplus \mathbb{R}^{n-1}$, hence it is isomorphic to $C(L)$, and we can apply Corollary \ref{tts:3}.
\end{proof}
%\medskip

Recall that two families $\mathcal{A}$ and $\mathcal{B}$ of infinite subsets of $\omega$ are \emph{separated} if there exists $S\subseteq\omega$ such that $A\sub^* S$, for each $A\in\mathcal{A}$, and $B\sub^* \omega\setminus S$ for each $B\in\mathcal{B}$.  Luzin constructed (in \textsf{ZFC}) an  almost disjoint family $\mathcal{L}$ of subsets of $\omega$ of cardinality $\omega_1$ such that no two disjoint uncountable subfamilies of $\mathcal{L}$ are separated, see \cite[Theorem 4.1]{vD84} and \cite{To96}
and references therein.

\begin{proposition}\label{AU:1} Let $\mathcal{A}$ be an  almost disjoint family of subsets of $\omega$ which contains two separated disjoint uncountable subfamilies. Then there exists an  $L\in \cde(K_\mathcal{A})$ such that there is no extension operator
$E:C(K_\mathcal{A})\to C(L)$ with $\|E\|<2$.
\end{proposition}

\begin{proof} Let $\mathcal{A}_0$ and $\cA_1$  be disjoint uncountable subfamilies of $\mathcal{A}$ separated by a set $S\subseteq\omega$. Without loss of generality we may assume that $\mathcal{A}_i$ have the cardinality $\omega_1$, so we can enumerate $\mathcal{A}_0\cup \mathcal{A}_1$ as $\{A_\alpha: \alpha<\omega_1\}$.
Let $\mathcal{L}=\{L_\alpha: \alpha<\omega_1\}$ be the Luzin almost disjoint family of subsets of a countable set $\omega'$.
We assume that the \textsf{AU}-compacta
\[ K_\mathcal{A}=\omega\cup\{p_A\colon A\in\mathcal{A}\} \cup \{\infty\} \mbox{ and }
K_\mathcal{L}=\omega'\cup\{r_L\colon L\in\mathcal{L}\} \cup \{\infty'\}\]
are disjoint. To simplify the notation we denote $p_{A_\alpha}$ by  $p_{\alpha}$ and $r_{L_\alpha}$ by  $r_{\alpha}$ for $\alpha<\omega_1$. Let $L'$ be the disjoint union of $K_\mathcal{A}$ and $K_\mathcal{L}$ and $L$ be the quotient space obtained from $L'$ by the identification of $p_{\alpha}$ with $r_{\alpha}$, for all $\alpha<\omega_1$, and $\infty$ with $\infty'$. Let $q:L'\to L$ be the quotient map. Clearly $q(K_\mathcal{A})$ is a topological copy of $K_\mathcal{A}$.

Suppose that there exists an extension operator
\[ E:C(q(K_\mathcal{A}))\to C(L) \mbox{ with }\|E\|=a<2.\]
Then by Lemma \ref{p:2}  there is a sequence $(\nu_p)_p$ in $M(q(K_\mathcal{A}))$ such that $\|\nu_p\|\le a$ for every $p\in \omega'$ and $\nu_p-\delta_{q(p)}\to 0$ in the $weak^*$ topology of $M(L)$. Let
\[\Gamma=\{\alpha<\omega_1: \nu_p(\{q(p_\alpha)\})\ne 0 \mbox{ for some } p\in \omega'\}.\]
 Obviously, the set $\Gamma$ is countable. We put
 \[ T=\{ p\in \omega': |\nu_p|(S)>a/2\}, \quad T'=\{ p\in \omega': |\nu_p|(\omega\setminus S)>a/2\};\]
 \[ \mathcal{L}_i= \{L_\alpha: A_\alpha\in\mathcal{A}_i,\ \alpha\in\omega_1\setminus\Gamma\} \mbox{ for } i=0,1.\]

  We will obtain the desired contradiction by showing that the set $T$ separates uncountable subfamilies $\mathcal{L}_i$ of $\mathcal{L}$.
  First, fix some $L_\alpha\in \mathcal{L}_0$. Take a finite set $F\sub\omega$ such that $A_\alpha\setminus F\sub S$.
  Note that the set $C=(A_\alpha\setminus F)\cup\{p_\alpha\}$ is clopen in $K_\mathcal{A}$ and  the set $D=L_\alpha \cup\{r_\alpha\}$ is clopen in $K_\mathcal{L}$. Therefore the characteristic function $f$ of $q(C\cup D)$ is continuous on $L$. For all $p\in L_\alpha$, we have $\delta_{q(p)}(f)=1$, so $(\nu_p)_p(f)\to 1$. Since $\nu_p(\{q(p_\alpha)\})=0$,  $\nu_p(A_\alpha\setminus F)>a/2$ for almost all $p\in L_\alpha$. It follows that $L_\alpha\sub^* T$. In the same way one can show that, for all $L_\alpha\in \mathcal{L}_1$,
$L_\alpha\sub^*T'$. It remains to observe that the assumption that $\|\nu_p\|\le a$ implies that $T'$ and $T$ are disjoint.
\end{proof}

\section{Remarks and open problems}\label{op}

Let us recall that a compact space is {\em Eberlein compact} if $K$ is homeomorphic to a weakly compact subset of a Banach space.
There are well-studied much wider classes of Corson and Valdivia compacta.

Given a cardinal number $\kappa$,  the $\Sigma$-product $\Sigma(\er^\kappa)$ of real lines is  the subspace of
$\mathbb{R}^\kappa$ consisting of functions with countable supports.
A compactum $K$ is  {\em Corson compact} if it can be
embedded into some $\Sigma(\er^\kappa)$;
$K$ is {\em Valdivia compact} if for some $\kappa$ there is an embedding $g:K\to\er^\kappa$ such that
$g(K)\cap\Sigma(\er^\kappa)$ is dense in the image,
see Negrepontis \cite{Ne84} and Kalenda \cite{Ka00}.

The following theorem summarizes known results concerning twisted sums of $c_0$ and $C(K)$ for these classes of compacta.

\begin{theorem}\label{Eb_Co_Va}
For a nonmetrizable $K$, there exists a nontrivial twisted sum of $c_0$ and $C(K)$ in any of the following cases:
\begin{itemize}
\item[(a)] $K$ is an Eberlein compact space (\cite{CGPY01});
\item[(b)] $K$ is a Corson compact space of weight $w(K)\ge\con$, in particular, under $\CH$, if $K$ is a Corson compact space (\cite[Theorem 3.1]{CT16});
\item[(c)] $K$ is a Valdivia compact space which does not satisfy ccc (cf. \cite{Ca16} and \cite{CT16}).
\end{itemize}
\end{theorem}

It is well-known that under $\MA$ and the negation of $\CH$, every Corson compact satisfying $ccc$ is metrizable. Hence, using  case (c) of the above theorem, we can state the second part of the result of Correa and Tausk from case (b) in a slightly stronger way:

\begin{theorem}[Correa and Tausk]
Assuming $\MA$, for every nonmetrizable Corson compact space $K$, there exists a nontrivial twisted sum of $c_0$ and $C(K)$.
\end{theorem}

It is natural to ask whether we can prove the above theorem in \textsf{ZFC}.

Let us note that part (c) of Theorem \ref{Eb_Co_Va} can be demonstrated as follows. If $K$ is Valdivia compact without $ccc$ then there is a retraction of $K$ onto
its subspace which has the weight $\omega_1$ and still does not satisfy $ccc$. Then one can apply Theorem \ref{Kubis}.
This suggests the following question.

\begin{problem}
Let $K$ be Valdivia compact that does not support a measure. Does  there exist a nontrivial twisted sum of $c_0$ and $C(K)$?
\end{problem}

The main obstacle here is that we do not know if every Valdivia compact space not supporting a measure has a retract of weight $\omega_1$ which
does not support a measure either.
\smallskip

We also recall a  related class of compact spaces:
a compactum $K$ is {\em Gul'ko compact} if $C(K)$ equipped with the weak topology is countably determined, i.e., is the continuous image of a closed subset of a product of some subset $S$ of the irrationals $P$ and a compact space (cf. \cite{Ne84}).
 We have the following relations between the classes of compacta mentioned above

\begin{center}
metrizable $\Rightarrow$ Eberlein $\Rightarrow$ Gul'ko $\Rightarrow$ Corson $\Rightarrow$ Valdivia
\end{center}

\noindent
and none of the above implications can be reversed, cf. \cite{Ne84}. Since each Gul'ko compact space satisfying $ccc$ is metrizable (cf. \cite[6.40]{Ne84}),
Theorem \ref{Eb_Co_Va}(c) yields

\begin{proposition}
For every nonmetrizable Gul'ko compact space $K$, there exists a nontrivial twisted sum of $c_0$ and $C(K)$.
\end{proposition}

Let us, finally, summarize the open problems mentioned in the  previous sections.
On one hand, we were not able to prove in Section \ref{tts} that
under $\MA(\omega_1)$ no separable compactum $K$ of weight $\omega_1$ admits a nontrivial twisted sum of $c_0$ and $C(K)$, see
 Problem \ref{tts:4}. On the other hand, our attempts at giving a \textsf{ZFC} construction of a separable compact space $K$ of weight $\omega_1$ and its
 countable discrete extension $L$ admitting no extension operator failed for some combinatorial reasons, see Problem \ref{d:5} and
 the assumption in Theorem \ref{I_B}.
 In all the cases we have considered,  one can construct such a pair $K$ and $L$ that there is no extension operator $E: C(K)\to C(L)$
 of small norm. However,   at each instance we needed some  additional set-theoretic assumption to kill all the possible extension operators, see e.g.\
 Proposition \ref{AU:1} and  Theorem \ref{auth}.  Therefore the following question is worth considering.

 \begin{problem}
 Does there exist a model of set theory in which   every twisted sum of $c_0$ and $C(K)$ is trivial whenever
 $K$ is a separable compactum of weight $\omega_1$?
\end{problem}

\appendix

\section{Bounded common extensions}

We discuss  here some consequences of a result due to Basile, Rao and Shortt \cite{BRS94} on common extensions of finitely additive signed measures.
Let $\fB$ be a Boolean algebras of subsets of $X$  and $\fB_1,\fB_2\sub\fB$ its two subalgebras. We consider $\nu_i\in M(\fB_i)$, $i=1,2$, where
the measures $\nu_1,\nu_2$ are  consistent, that is
$\nu_1(B)=\nu_2(B)$ for every $B\in\fB_1\cap\fB_2$.

Let $\eta$ be a function defined on $\fB_1\cup\fB_2$ by $\eta(B)=\nu_1(B)$ for $B\in\fB_1$ and $\eta(B)=\nu_2(B)$ for
$B\in\fB_2$. We define
\[SC(\nu_1,\nu_2)=\sup\left\{ \sum_{i=1}^n \left| \eta(B_{i+1})-\eta(B_i)\right|  \right\},\]
where the supremum is taken over all $n\ge 0$ and all increasing chains $\emptyset=B_0\sub B_1\sub\ldots\sub  B_{n+1}=X$ from
$ \fB_1\cup \fB_2$.

Theorem 1.5 from \cite{BRS94} asserts that there is a common extension of $\nu_1,\nu_2$ to a measure $\lambda\in M(\la\fB_1\cup\fB_2\ra)$ such that $\|\lambda\|=SC(\nu_1,\nu_2)$. Clearly, we can extend such $\lambda$ to $\fB$ preserving its norm.

\begin{lemma}\label{ap:1}
Let $\fB$ be a finite algebra having $N$ atoms. Suppose that
$\fB_1,\fB_2\sub\fB$ are subalgebras,  $\nu_i\in M(\fB_i)$ for  $i=1,2$ are two consistent measures, and $\delta>0$.

\begin{enumerate}[(a)]
\item If $|\nu_i(B)| < \delta$ for $B\in\fB_i$, $i=1,2$, then there is a common extension of $\nu_1,\nu_2$ to $\lambda\in M(\fB)$ such that
$\|\lambda\|\le 2N\delta$.
\item If $\lambda\in M(\fB)$ is such a measure that $|\lambda(B)-\nu_i(B)| < \delta$ for $B\in\fB_i$, $i=1,2$, then
there is a common extension of $\nu_1,\nu_2$ to $\lambda'\in M(\fB)$ such that
$\|\lambda-\lambda'\|\le 2N\delta$.
\item In the setting of (b), if moreover $\nu_1,\nu_2$ and $\lambda$ have rational values then there is such $\lambda'$ that also assumes only rational values.
\end{enumerate}
\end{lemma}

\begin{proof}
To check $(a)$ it is enough to notice that if $B_0, \ldots, B_{n+1}$ is a strictly increasing  chain in $\fB_1\cup\fB_2$ then $n+1\le N$ so clearly $SC(\nu_1,\nu_2)\le 2N\delta$.

To get $(b)$ we can apply $(a)$ to the measures $\nu_1'=\nu_1-\lambda$ and $\nu_2'=\nu_2-\lambda$ considered on $\fB_1$ and $\fB_2$,
respectively.

For $(c)$ we may also assume that $\delta\in\qu$.
By (b) the set
\[ E=\{\mu\in M(\fB): \mu\mbox{ extends } \nu_1,\nu_2\mbox{ and } \|\mu-{\lambda}\|\le 2N\delta\},\]
is nonempty. The set $E$ may be identified with a symplex in $\er^N$ defined by equations and inequalities with rational coefficients.
Hence any extreme point of $E$ has rational coefficients and defines the required measure $\lambda'$.
\end{proof}

\begin{lemma}\label{ap:2}
Let $\fB$ be a finite algebra. % having $N$ atoms.
For any  subalgebra $\fC\sub\fB$  and a measure $\nu \in M(\fC)$,
if, for some $\delta>0$, there is $\lambda\in M_1(\fB)$ such that $|\lambda(C)-\nu(C)|<\delta$ for $C\in\fC$ then
there is an extension of $\nu$ to $\mu\in M(\fB)$ such that $\|\mu\| \le \max(1,\|\nu\|)$
and $|\mu(B)-\lambda(B)|\le 3\delta$ for every $B\in\fB$.

If, moreover, $\nu$ and $\lambda$ have rational values then there is such $\mu$ with rational values.
\end{lemma}

\begin{proof}
Note first that for any $\nu_1\in M(\fC)$, if $|\nu_1(C)|<\delta$ for every $C\in\fC$ then there is an extension
$\widetilde{\nu_1}\in M(\fB)$ of $\nu_1$ such that $|\widetilde{\nu_1}(B)|<\delta$ for every $B\in\fB$. Indeed, we can define
such $\widetilde{\nu_1}$ by the following procedure:
If $C$ is an atom of $\fC$ then choose any atom $B$ of $\fB$ contained in $C$ and set $\widetilde{\nu_1}(B)=\nu_1(C)$ and
$\widetilde{\nu_1}(B_1)=0$ for every $B_1\in\fB$ contained in  $C\sm B$.
Note also that then $\widetilde{\nu_1}$ satisfies $\|\widetilde{\nu_1}\|<2\delta$.

We can now apply the  preceding remark to $\nu_1=\nu-\lambda$ considered on $\fC$ to get $\widehat{\nu_1}$ as above.
Then the measure $\lambda'=\widehat{\nu_1}+\lambda$ extends $\nu$ and satisfies $\|\lambda'\|<\|\lambda\|+2\delta\le 1+2\delta$.

Now it is enough to check that we can appropriately lower the size of $\|\lambda'\|$.
Consider first some atom $C$ of $\fC$ and let  $B^+$ be the union of all atoms $B$ of $\fB$ contained in $C$ for which
$\lambda'(B)>0$; set  $B^-=C\sm B^+$. Note that if $t\le \min(|\lambda'(B^+),|\lambda'(B^-)|$ then we can modify $\lambda'$ on $C$,
assigning the value $\lambda'(B^+)-t$ to $B^+$ and $\lambda'(B^-)+t$ to $B^-$; this  defines  an extension of $\nu$ of norm $\|\lambda'\|-2t$.

Let now $C_1,\ldots, C_m$ be the list of all atoms; we divide every $C_j$ into $B_j^+$ and $B_j^-$ as described above.
Let
\[ p=\sum_{j\le m} \min\big( \lambda'(B_j^+),|\lambda'(B_j^-)|\big).\]
 If $p\ge\delta$, by the procedure described above we
shall get a measure $\mu$ extending $\nu$ with $\|\mu\|\le 1$. Namely,  we then choose numbers nonnegative $t_j\le  \min\big( |\lambda'(B_j^+),|\lambda'(B_j^-)|\big)$
such that $\sum_{j\le m} t_j=\delta$, and apply the modification by $t_j$ to $C_j$.
 If $p<\delta$ then the same procedure will give
$\mu$ such that $\mu$ is either nonnegative or nonpositive on each $C_i$. In such a case
\[\|\mu\|=\sum_{j\le m} |\mu(C_j)|=\sum_{j\le m} |\nu(C_j)|=\|\nu\|.\]
In both cases we shall have $|\mu(B)-\lambda'(B)|\le 2\delta$ for any $B\in\fB$ so $\mu$ will be the required  measure.

%Suppose that $\eta=\min(\|\mu\|-1, \|\mu\|-\|nu\|)>0$; then $\eta<2\delta$.
%Consider an atom $C$ of $\fC$ such that $|\mu|(C)>|\nu|(C)=|\nu(C)|$. Suppose that, for instance, $\mu(B^+)\ge -\mu(B^-)$.
%We have
%\[\mu(B^+)-\mu(B^-)=|\mu|(C), \mu(B^+)+\mu(B^-)=\mu(C)=\nu(C),\]
%so $-2\mu(B^-)=|\mu|(C)-\nu(C)<2\delta$, that is $|\mu(B^-)|<\delta$. W can change $\mu$ on $C$ and define $\mu'$ do that
%$\mu'(B)=0$ for every atom $B$ of $\fB$ contained in $B^-$ preserving $\mu'(C)=\mu(C)$. Then $|\mu'|(C)=|\nu(C)|$

For the final statement just note that we can assume that $\delta\in\qu$; the above argument shows that
in such a case $\lambda'$ and $\mu$ have values in $\qu$.
\end{proof}

In the last auxiliary result we consider the the following set in a Euclidean space:
\[T(a,b)=\left\{x\in \er^m\times\er^n: \sum_{j\le n}x_{ij}=a_i \mbox{ for  } i\le m,  \sum_{i\le m}x_{ij}=b_j \mbox{ for } j\le n\right\}.\]

\begin{lemma}\label{ap:3}
For and $a\in \er^m$, $b\in \er^n$ such that $\sum_{i\le m} a_i=\sum_{j\le n} b_j$
 there is $x\in T(a,b)$ satisfying
 \[\sum_{i,j}|x_{ij}|\le \max\big( \sum_{i\le m} |a_i|, \sum_{j\le n} |b_j|\big).\]
\end{lemma}

\begin{proof}
The assertion is clearly true of either $m=1$ or $n=1$. We argue by induction on $m+n$.

Since $\sum_{i\le m} a_i=\sum_{j\le n} b_j$ there are $i$ and $j$ such that $a_i$ and $b_j$ have the same sign. Suppose e.g.\
that this is the case for $i=j=1$. Moreover, let us assume that $0\le a_1\le b_1$; the other case may be treated by symmetric argument.
Set

\begin{enumerate}[(i)]
\item $x_{11}=a_1$ and $x_{1,j}=0$ for $j>1$;
\item $b_1'=b_1-a_1$, $b_j'=b_j$ for $j>1$;
\item $a'=(a_2,\ldots, a_m)\in \er^{m-1}$.
\end{enumerate}

Then for
\[r'=\max\big( \sum_{2\le i\le m} |a_i|, \sum_{j\le n} |b_j'|\big ),\]
 by the inductive assumption there is
 \[x'=(x_{ij})_{2\le i\le m, 1\le j\le n},\]
 such that $x'\in T(a',b')$ and $\|x'\|\le r'$.
 Note that
 \[\sum_{2\le i\le m} |a_i|\le r-a_1,\]
 \[ \sum_{j\le n} |b_{j}'|=b_1-a_1  +\sum_{j\le 2\le n} |b_j|\le  b_1-a_1+ r- b_1=r-a_1,\]
 so $r'\le r-a_1$. Hence we can extend $x'$ by the first row defined above and get  the required vector $x$.
 \end{proof}


\begin{thebibliography}{10}

\bibitem{AU29} P.S. Aleksandrov and P. Urysohn, \emph{M\'emoire sur
les espaces topologiques compacts}, Verh.\ Akad.\ Wetensch., Amsterdam 14 (1929).

\bibitem{AT11}
 A.\ Avil\'es and  S.\ Todorcevic, {\em Multiple gaps},  Fund.\ Math.\ 213 (2011), 15--42.
\bibitem{ASCGM}
A.\ Avil\'es, F.\ Cabello S\'anchez, J.M.F.\ Castillo, M.\ Gonz\'alez, Y.\ Moreno,
{\em Separably injective Banach spaces},
Lecture Notes in Mathematics 2132, Springer (2016).


\bibitem{BRS94}
A.\ Basile, K.P.S.\  Bhaskara Rao, R.M.\  Shortt, {\em Bounded common extensions of charges},
Proc.\ Amer.\ Math.\ Soc.\ 121 (1994),  137--143.

\bibitem{CCKY03}
F.\ Cabello S\'anchez, J.M.F.\  Castillo, N.J.\  Kalton, D.T.\  Yost, {\em Twisted sums with $C(K)$ spaces},
Trans.\ Amer.\ Math.\ Soc.\ 355 (2003),  4523--4541.

\bibitem{CCY00}
F.\ Cabello S\'anchez, J.M.F.\  Castillo, %N.J.\  Kalton,
D.T.\  Yost, {\em Sobczyk's Theorem from A to B}, Extracta Math.\ 15(2) (2000), 391--420.

\bibitem{Ca16}
J.M.F.\ Castillo, {\em Nonseparable $C(K)$-spaces can be twisted when K is a finite height compact},  Topology Appl.\ 198 (2016), 107--116.

\bibitem{CGPY01} J.M.F.\ Castillo, M.\ Gonz\'alez, A.\ Plichko, D.\ Yost, {\em Twisted properties of Banach spaces}, Math.\ Scand.\ 89 (2001) 217--244.

\bibitem{CT15}
C.\ Correa and D.V.\ Tausk, {\em  Extension property and complementation of isometric copies of continuous functions spaces}, Results Math. 67 (2015), 445--455.

\bibitem{CT16}
C.\ Correa and D.V.\ Tausk, {\em  Nontrivial twisted sums of $c_0$ and $C(K)$},  J.\ Funct.\ Anal.\ 270 (2016),  842--853.

 \bibitem{DH76} S.\ Ditor and R.\ Haydon, {\em On absolute retracts, $P(S)$, and complemented subspaces of $C(D^{\omega_1})$},  Studia Math.\ 56 (1976),
  243--251.

 \bibitem{vD84} E.~van Douwen, {\em The integers and topology}, in: {\em Handbook of set-theoretic topology}, K.\ Kunen, J.\ Vaughan (eds.),  North-Holland, Amsterdam (1984), 111--167.

 \bibitem{DP16}
 P.\ Drygier and G.\ Plebanek, {\em Compactifications of $\omega$ and the Banach space $c_0$}, Fund.\ Math.\ 237 (2017), 165--186.

%\bibitem{Ef65} B.A.\ Efimov,  {\em Dyadic bicompacta} (Russian), Trudy Moskov.\
%Mat.\ Ob\v s\v c.\ 14 (1965), 211--247.

 \bibitem{En89}
 R.~Engelking, \emph{General Topology}, Heldermann Verlag, Berlin (1989).

\bibitem{Fr84}
D.H.\ Fremlin,  {\em Consequences of Martin's axiom},
Cambridge Tracts in Mathematics  84, Cambridge University Press, Cambridge (1984).

 \bibitem{Ha74} R.\ Haydon, {\em On a problem of Pe\l czy\'{n}ski: Milutin spaces, Dugundji spaces and AE(0-dim)},  Studia Math. 52 (1974), 23--31.

 \bibitem{HL74} R.W.\ Heath and D.J.\ Lutzer, {\em Dugundji extension theorems for linearly ordered spaces}, Pacific J.\ Math.\ 55 (1974), 419--425.

\bibitem{Je03}
T.\  Jech, {\em Set theory. The third millennium edition, revised and expanded},   Springer-Verlag, Berlin (2003).


\bibitem{Ka00}
O.\ Kalenda, {\em Valdivia compact spaces in topology and Banach space theory}, Extracta Math.\ 15 (2000), 1-85.

\bibitem{KKL11}
J.\ K\c{a}kol, W.\  Kubi\'s, M.\  L\'opez-Pellicer, {\em Descriptive topology in selected topics of functional analysis},
Developments in Mathematics 24,  Springer, New York (2011).

\bibitem{Ko05} P.\ Koszmider, {\em On decompositions of Banach spaces of continuous functions on Mr\'owka's spaces},
Proc.\ Amer.\ Math.\ Soc.\ 133 (2005),  2137–-2146.

\bibitem{Ku04} W.\ Kubi\'s, {\em A representation of retracts of cubes},  	arXiv:math/0407196.

 \bibitem{MP09} W.\ Marciszewski and R.\ Pol, \emph{On Banach spaces whose norm-open sets are
 $F_\sigma$-sets in the weak topology}, J.\ Math.\ Anal.\ Appl.\
 {350} (2009), 708--722.

\bibitem{Me96} S.\  Mercourakis, {\em Some remarks on countably determined measures and uniform distribution of sequences},
 Monatsh.\ Math.\  121 (1996), 79--111.

\bibitem{Ne84} S.\ Negrepontis, {\em Banach spaces and topology}, in: {\em Handbook of set-theoretic topology}, K.\ Kunen, J.\ Vaughan (eds.),  North-Holland, Amsterdam (1984), 1045--1142.

 \bibitem{Os74}
 A.J.\ Ostaszewski, {\em A characterization of compact,
 separable, ordered spaces}, J.\ London Math.\ Soc.\ 7 (1974), 758--760.

 \bibitem{Sa80} A.\ Sapounakis,  {\em Measures on totally ordered spaces},  Mathematika 27 (1980),  225--235.

 \bibitem{Sh76} E.V.\ Shchepin, {\em Topology of limit spaces with uncountable inverse spectra}, Uspekhi Mat. Nauk, 31 (1976), no. 5 (191), 191--226 (Russian Mathematical Surveys, 1976, 31:5, 155--191).
\bibitem{To96}
S.\ Todor\v{c}evi\'c,   {\em Analytic gaps}, Fund.\ Math.\ 150 (1996),  55--66.

\end{thebibliography}
\end{document}